\documentclass[11pt,a4paper]{article}
\usepackage{amsmath, amsthm, amsfonts, amssymb, color}
\usepackage{mathrsfs}
\usepackage{color}
\usepackage{bbm}
\usepackage[notcite,notref]{showkeys}

\newtheorem{thm}{Theorem}[section]
\newtheorem{cor}[thm]{Corollary}
\newtheorem{lem}[thm]{Lemma}
\newtheorem{prp}[thm]{Proposition}

\newtheorem{exa}[thm]{Example}
\newtheorem{rem}[thm]{Remark}
\theoremstyle{definition}
\newtheorem{defn}{Definition}[section]
\newcommand{\scr}[1]{\mathscr #1}
\definecolor{wco}{rgb}{0.5,0.2,0.3}

\numberwithin{equation}{section} \theoremstyle{remark}

\newcommand{\ua}{\uparrow}

\title{{\bf Propagation of Chaos for Derivatives of McKean--Vlasov Stochastic Differential Equations and Applications   }
	\footnote{Supported in part by   the National Key R\&D Program of China (No. 2022YFA1006000).} }
\author{
	{\bf  Xiao-Yu Zhao  }\\
	\footnotesize{ Center for Applied Mathematics and KL-AAGDM, Tianjin University,   300072, China}\\
	\footnotesize{ zhxy\_0628@tju.edu.cn}}
\begin{document}
	\allowdisplaybreaks
	\def\R{\mathbb R}  \def\ff{\frac} \def\ss{\sqrt} \def\B{\mathbf
		B}\def\TO{\mathbb T}
	\def\I{\mathbb I_{\pp M}}\def\p<{\preceq}
	\def\N{\mathbb N} \def\kk{\kappa} \def\m{{\bf m}}
	\def\ee{\varepsilon}\def\ddd{D^*}
	\def\dd{\delta} \def\DD{\Delta} \def\vv{\varepsilon} \def\rr{\rho}
	\def\<{\langle} \def\>{\rangle} \def\GG{\Gamma} \def\gg{\gamma}
	\def\nn{\nabla} \def\pp{\partial} \def\E{\mathbb E}
	\def\d{\text{\rm{d}}} \def\bb{\beta} \def\aa{\alpha} \def\D{\scr D}
	\def\si{\sigma} \def\ess{\text{\rm{ess}}}
	\def\beg{\begin} \def\beq{\begin{equation}}  \def\eed{\end{equation}}\def\F{\scr F}
	\def\Ric{{\rm Ric}} \def\Hess{\text{\rm{Hess}}}
	\def\e{\text{\rm{e}}} \def\ua{\underline a} \def\OO{\Omega}  \def\oo{\omega}
	\def\tt{\tilde}
	\def\cut{\text{\rm{cut}}} \def\P{\mathbb P} \def\ifn{I_n(f^{\bigotimes n})}
	\def\C{\scr C}      \def\aaa{\mathbf{r}}     \def\r{r}
	\def\gap{\text{\rm{gap}}} \def\prr{\pi_{{\bf m},\varrho}}  \def\r{\mathbf r}
	\def\Z{\mathbb Z} \def\vrr{\varrho} \def\ll{\lambda}
	\def\L{\scr L}\def\Tt{\tt} \def\TT{\tt}\def\II{\mathbb I}
	\def\i{{\rm in}}\def\Sect{{\rm Sect}}  \def\H{\mathbb H}
	\def\M{\scr M}\def\Q{\mathbb Q} \def\texto{\text{o}} \def\LL{\Lambda}
	\def\Rank{{\rm Rank}} \def\B{\scr B} \def\i{{\rm i}} \def\HR{\hat{\R}^d}
	\def\to{\rightarrow}\def\l{\ell}\def\iint{\int}
	\def\EE{\scr E}\def\Cut{{\rm Cut}}\def\W{\mathbb W}
	\def\A{\scr A} \def\Lip{{\rm Lip}}\def\S{\mathbb S}
	\def\BB{\mathbb B}\def\Ent{{\rm Ent}} \def\i{{\rm i}}\def\itparallel{{\it\parallel}}
	\def\g{{\mathbf g}}\def\Sect{{\mathcal Sec}}\def\T{\mathcal T}\def\V{{\bf V}}
	\def\PP{{\bf P}}\def\HL{{\bf L}}\def\Id{{\rm Id}}\def\f{{\bf f}}\def\cut{{\rm cut}}
	\def\Ss{\mathbb S}
	\def\BL{\scr A}\def\Pp{\mathbb P}\def\Pp{\mathbb P} \def\Ee{\mathbb E} \def\8{\infty}\def\1{{\bf 1}}
	\maketitle
	\def\Cc{\mathcal C} \def\t{\theta}
	\begin{abstract}  
	As an enhanced version of existing results on Kac's propagation of  chaos, which describes  the convergence of  mean-field particle systems to a system of independent McKean--Vlasov particles as the number of particles tends to infinity, we prove the convergence at the level of derivatives with respect to the perturbations of the initial values and the driving noises,  together with explicit convergence rates that can be sharp. 
	As a consequence, the intrinsic derivative with respect to  the initial distribution  of  a single particle converges to that of an independent McKean--Vlasov SDE.
		
	\end{abstract} 
	
	\section{Introduction}

Kac's chaotic property, also called the Boltzmann property, is a key concept for deriving the spatially homogeneous Boltzmann equation; see \cite{Kac}.
Its dynamical version, known as propagation of chaos, shows that in a mean-field system the limit of a single particle is described by a McKean--Vlasov stochastic differential equation (SDE), as proposed in \cite{Mckean}; see also \cite{SN}. Such equations, also called distribution dependent or mean-field SDEs, are important probabilistic models for nonlinear Fokker--Planck equations and mean field games; see, for example, \cite{CD,RW24,SN}. The theory of propagation of chaos has been extensively developed; see the surveys and monographs \cite{CD,G16,H91,SN} and representative
quantitative results in, for instance, \cite{BH,BJW,SWT,WHGY}.

The standard propagation of chaos concerns the convergence of the solution processes themselves. It gives rates for the approximation of the McKean--Vlasov solution by a finite particle system. Apart from the solution processes,  derivative estimates are also important in stochastic analysis.  In this paper, we study  this problem at the level of derivatives. More precisely, we ask whether propagation of chaos remains valid after differentiating the dynamics. Our results show that propagation of chaos holds not only for the solutions, but also for the associated directional derivative flows, Malliavin derivative flows, and semigroup derivatives. Moreover, we obtain explicit convergence rates for these derivative-level
approximations. 
To the best of our knowledge, quantitative propagation of chaos at the level of derivatives was not previously available. A notable feature of our estimates is that, in certain cases, the rates coincide with the sharp empirical-measure rates appearing in the usual propagation of chaos.

Derivative estimates have long been studied for classical SDEs. They concern, for instance, directional derivatives of the solution flow with respect to the initial point, as well as Malliavin derivatives with respect to the driving noise. A classical application of such estimates is the Bismut formula, a derivative formula which gives a probabilistic representation for derivatives of diffusion semigroups.   The formula was first established by Bismut \cite{Bismut} via Malliavin calculus and was later reproved by Elworthy and Li \cite{EL} through a martingale approach. Thus, it is also called the Bismut--Elworthy--Li formula. 
More recently, the Bismut formula and the related derivative estimates have been extended to SDEs with singular coefficients; see Zhang \cite{Z16} for  integrable coefficients  and \cite{XXZZ} for  localized integrability coefficients.
In addition, Menoukeu et al. \cite{MNPZ} studied differentiability properties of strong solutions with bounded measurable drifts, in particular Malliavin differentiability.

For McKean--Vlasov SDEs, directional derivative flows and related Malliavin derivative flows are also important analytic tools in the regularity theory. For regular distribution dependent SDEs,  Ren and Wang \cite{RW19} studied these two derivative flows in order to establish the Bismut formula for Lions derivatives. The path-distribution dependent case was treated by  \cite{BRW20}. For singular McKean--Vlasov SDEs,  \cite{W23b} developed derivative estimates and established a Bismut formula for intrinsic derivatives.

The results mentioned above consider derivative estimates and derivative formula for SDEs or McKean--Vlasov SDEs themselves. The question studied
in the present paper is different: whether these derivative flows can be approximated by the derivative flows of weakly interacting particle systems. This derivative level propagation of chaos is not a consequence of the usual propagation of chaos for the solution processes, and  to the best of our knowledge, it has not been systematically studied before.

The closest existing works are \cite{dRW} and \cite{BdRW}, both of which study related questions  from the viewpoint of Malliavin calculus and particle systems. In \cite{dRW}, Malliavin differentiability was proved for McKean--Vlasov SDEs, and an interacting particle system was used to transfer Malliavin regularity to the limiting equation. In the common noise setting, \cite{BdRW} proved Malliavin differentiability with respect to the common noise, identified the Malliavin derivative of the conditional law, and derived an integration by parts formula on the
Wiener space. These results presented a qualitative propagation of chaos result for Malliavin derivatives. 

Compared with  the existing studies on the  propagation of chaos for  distributions and  Malliavin derivatives,   the present paper focuses on the propagation of chaos for  derivative flows and Bismut type derivative formulas. More precisely: 
\begin{itemize}
	\item[(i)] \emph{Propagation of chaos is derived for two types of derivative flows,} which describe sensitivities with respect to the initial distribution and the driving noise, and   which are the basic objects appearing in intrinsic derivatives on the Wasserstein space; see, for instance \cite{RW19,W23b}.
	To this end, we first study directional derivative
	flows.  Then   we consider the Malliavin derivatives along directions associated with the directional derivative flows (see \eqref{LLP1}, \eqref{LD00} and \eqref{LD0} below), which are  different from  those studied in  \cite{BdRW,dRW}.   
	\item[(ii)] \emph{A finite-particle approximation is presented for  the Bismut  derivative formula.}   As an application of the  estimates for derivative flows, we show in
	Section~\ref{app} that this approximation can be lifted to the level of intrinsic derivatives with respect to initial distribution; namely, the intrinsic derivative of the finite-particle semigroup converges to that of the limiting nonlinear semigroup, and the corresponding Bismut formula admits a quantitative finite-particle approximation. This is essentially different from the existing   propagation of chaos results at the level of distributions.  	
	\item[(iii)] \emph{Explicit convergence rates are obtained.} The propagation of chaos results related to Malliavin derivatives in \cite{BdRW,dRW} are qualitative, but no explicit rate is given there for the convergence of derivative flows. In this paper, we prove quantitative convergence for both directional derivative flows and the corresponding Malliavin derivative flows. The rates are determined by the regularity of the first-order derivatives of the
	coefficients, the moment order of the initial data and perturbations, and the empirical-measure approximation rate. 
	In particular, when the first-order derivatives of the coefficients are globally Lipschitz and the initial data and perturbations are bounded, these rates match the sharp empirical-measure rate underlying the usual propagation of chaos. Hence the derivative-level rates are sharp in this case, see Remark~\ref{sharp}(2) for a more detailed discussion.
	
	In addition, the propagation of chaos for the intrinsic derivative with respect to the initial distribution is also quantitative. For test functions in $C_b^2(\mathbb R^d)$, we obtain explicit convergence rates for this
	intrinsic derivative.
	Since this intrinsic derivative can be represented by the Bismut formula, this yields a quantitative finite-particle approximation of the Bismut formula.
\end{itemize}
Hence, the usual propagation of chaos is extended quantitatively to the derivative level.

We now describe the setting of the paper. Throughout the paper, let	
$(\R^d, |\cdot |)$ be the $d$-dimensional Euclidean space for some $d\in\mathbb N$.
Denote by $\R^d\otimes\R^d$   the family of all $d\times d$-matrices
with real entries, which is  equipped with the operator norm $\|\cdot\|$.    Let  $A^*$ denote the   transpose of $A\in \R^d\otimes \R^d$, and let
$\|\cdot\|_{\infty}$ be the uniform norm for functions taking values in $\R,\R^d$ or $\R^d\otimes \R^d$.
Let  $\|\cdot\|_{\rm HS}$ denote the Hilbert-Schmidt norm of matrices and $\scr P$ be the space of probability measures on $\R^d$. For any $k\in [2,\infty)$, the $L^k$-Wasserstein space
$$\scr P_k:=\{\mu\in \scr P:   \mu(|\cdot|^k)<\infty\}$$ is a Polish space
under  the  $L^k$-Wasserstein distance
$$\W_k(\mu,\nu):=   \inf_{\pi\in \C(\mu,\nu)} \bigg(\int_{\R^d\times\R^d} |x-y|^k \pi(\d x,\d y) \bigg)^{\ff 1{k}},\ \ \mu,\nu\in \scr P_ k,$$
where $\C(\mu,\nu)$ is the set of all couplings for $\mu$ and $\nu$.

For fixed $T>0$, consider the McKean--Vlasov SDE on $\R^d$
\begin{align}\label{E0}
	\d X_t=b_t(X_t,\L_{X_t})\d t+\sigma_t(X_t,\L_{X_t})\d W_t, \ \ t\in[0,T], 
\end{align}
where $W_t$ is an $m$-dimensional Brownian motion on  $(\OO, \{\F_t\}_{t\ge 0},\P)$, which is a complete filtered  probability space,  $\L_\xi$ is the distribution (i.e. the law) of a random variable $\xi$,
and \beg{align*}b:  [0,T]\times\R^d\times \scr  P_k\to \R^d,
\ \ \si: [0,T]\times\R^d\times \scr  P_k \to \R^d\otimes\R^m\end{align*}  are measurable. 
When different probability measures are considered, we denote $\L_{\xi}=\L_{\xi|\P}$ to emphasize the distribution of $\xi$ under $\P$.

Since \eqref{E0} is distribution dependent, weakly interacting particle systems are commonly used to approximate it. Let $N\ge1$ be an integer, and $(X_0^i,\eta^i,W^i_t)_{1\leq i\leq N}$ be independent and identically distributed (i.i.d.) copies of $(X_0,\eta,W_t)$, where $\eta,\eta^i$ will be used as perturbations of the initial values.  We also have the following non-particle system  McKean--Vlasov SDEs:
\begin{align}\label{E1}\d X_t^i= b_t(X_t^i, \L_{X_t^i})\d t+  \sigma_t(X^i_t,\L_{X_t^i}) \d W^i_t,\ \ t\in[0,T],\,1\leq i\leq N.
\end{align}
Denote $$\tt\mu_t^N:=\ff 1 N\sum_{j=1}^N\dd_{X_t^j}$$ as the empirical distribution of $(X_t^i)$. 
Consider the mean-field interacting particle system:
\begin{align}\label{E2}
\d X^{i,N}_t=b_t(X_t^{i,N}, \hat\mu_t^N)\d t+\sigma_t(X^{i,N}_t, \hat\mu_t^N) \d W^i_t,\ \ t\in[0,T], \,\,X_0^{i,N}=X_0^i,\,1\leq i\leq N,
\end{align}
where $$\hat\mu_t^N:=\ff 1 N\sum_{j=1}^N\dd_{X_t^{j,N}}.$$

The main results of this paper are Theorems~\ref{gf2} and~\ref{gf3}, which establish quantitative propagation of chaos with explicit convergence rates from the interacting particle system \eqref{E2} to the independent McKean--Vlasov system \eqref{E1} at the derivative level. 
More precisely, Theorem~\ref{gf2} treats the directional derivative flows, while Theorem~\ref{gf3} focuses on the corresponding Malliavin derivative flows. 
These estimates further yield a quantitative finite-particle approximation for intrinsic derivatives with respect to the initial distribution.

The rest of the paper is organized as follows. Section~\ref{pre} collects preliminary results, including the basic propagation of chaos estimates for the solution processes and the well-posedness of the derivative flows.
Section~\ref{apd} contains the main results, where quantitative propagation of chaos is proved for the directional and Malliavin derivative flows. Section~\ref{app} is devoted to the application to intrinsic derivatives.
\section{Preliminaries}\label{pre}	
\subsection{Classical results}
We briefly recall classical results involving the relationship between the independent McKean--Vlasov system \eqref{E1} and the interacting particle system \eqref{E2}.
\begin{enumerate}
\item[\bf (H)] Let $k\geq2$. There exists a constant $K>0$ such that $b$ and $\si$ satisfy that
\begin{equation}\label{SPPn}
	\sup_{t\in[0,T]}\left\{ |b_t({\bf 0},\dd_{{\bf 0}})|+\|\si_t({\bf 0},\dd_{{\bf 0}})\|\right\}<\8,\ \ 
\end{equation}
and
\begin{equation*}\label{SPP1}\beg{split} &2\<x-y,  b_t(x,\mu)- b_t(y,\nu)\>^++\|\si_t(x,\mu)-\si_t(y,\nu)\|_{HS}^2\\
	&\le K\big(|x-y|^2+\W_k(\mu,\nu)^2\big),\ \ t\in[0,T],\, x,y\in \R^d,\, \mu,\nu\in \scr P_k(\R^d),
	\end{split}\end{equation*}
	here and in what follows, for $x\in\R^d$, we denote  by $\dd_{x}$   the Dirac measure at $x$.
\end{enumerate} 
According to \cite[Theorem 2.1]{W18}, under  {\bf (H)},  for any initial values $X_0^i\in L^k(\OO\to \R^d,\F_0,\P)$, \eqref{E1}  has a unique solution $(X_t^i)_{1\leq i\leq N}$. Moreover,   for any $m\geq1$, there exists a constant $c(m)>0$ such that
\begin{equation}\label{MM1}
\E \Big[\sup_{t\in [0,T]} |X_t^i|^m\Big|\F_0\Big]\leq c(m)(1+|X_0^i|^m),\ \ 1\leq i\leq N.
\end{equation}
Since $(X_0^i,W^i)_{1\le i\le N}$ are i.i.d.
and \eqref{E1} is well-posed, the processes \((X_t^i)_{1\le i\le N}\) are i.i.d. We write  $$\mu_t:=\L_{X_t^i},\ \ t\in[0,T],\,1\leq i\leq N,$$ which is independent of $i$. 

The following lemma plays a vital role in the approximation of derivative  flows. For the case $k=2$, the corresponding convergence results can be found in  \cite[Lemma 1.9 and Theorem 1.10]{C16}. For completeness, we provide
the proof in Appendix \ref{2.1}.
\begin{lem}\label{poc}
Assume {\bf (H)}. Then, for the i.i.d. initial values $(X_0^i)_{1\leq i\leq N}$ introduced above with $X_0\in L^k(\OO\to \R^d,\F_0,\P)$,  \eqref{E2}  has a unique solution $(X_t^{i,N})_{1\leq i\leq N}\in L^k(\OO\to \R^d,\F_0,\P)$, which satisfies that for any $m\geq k$, there exists a constant $c(m)>0$ such that
\begin{equation}\label{MM2}
\E \Big[\sup_{t\in [0,T]} |X_t^{i,N}|^m\Big|\F_0\Big]\leq c(m)\bigg(1+|X_0^i|^m+\ff 1 N\sum_{j=1}^N|X_0^j|^m\bigg),\ \ 1\leq i\leq N.
\end{equation}
Consequently, 	\begin{equation}\label{MM3}
\E \Big[\sup_{t\in [0,T]} |X_t^{i,N}|^k\Big]\leq 2c(k)\bigg(1+\E|X_0|^k\bigg)<\8,\ \ 1\leq i\leq N.
\end{equation}
Moreover, 
\begin{align}\label{ap01}
\lim_{N\to\8}\left(
\E\left[\sup_{t\in[0,T]}|X_t^{1,N}-X_t^1|^k\right]+\sup_{0\leq  t\leq T}\E\left[\W_k(\hat\mu_t^N,\mu_t)^k\right]\right)=0.
\end{align}
If further assume that $\mu=\L_{X_0}\in\scr P_{q}$ for some $q>k$, then we can find a constant $C>0$ independent of $N$, such that
\begin{align}\label{ap02}
\E\left[\sup_{t\in[0,T]}|X_t^{1,N}-X_t^1|^k\right]+\sup_{0\leq  t\leq T}\E\left[\W_k(\hat\mu_t^N,\mu_t)^k\right]\leq C\epsilon(N),
\end{align}
where \begin{align}\label{en}\epsilon(N):=\begin{cases}
	N^{- 1 /2}+N^{-(q-k)/q}, &\textrm{if } k>d/2\, \textrm{and} \,q\ne 2k\\
	N^{- 1 /2}  \log (1+N)+N^{-(q-k)/q}, & \textrm{if } k=d/2\, \textrm{and} \,q\ne 2k,\\
	N^{-k/{d}}+N^{-(q-k)/q}, & \textrm{if } k\in(0,d/2)\, \textrm{and}\,q\ne d/(d-k).\end{cases}
	\end{align}
\end{lem}
\begin{rem}\label{exc}
Since  $(X_0^i,W^i_{\cdot})_{1\leq i\leq N}$ are i.i.d. and the particle system \eqref{E2} is symmetric in the particle labels, the uniqueness in Lemma \ref{poc} implies that $(X_t^{i,N})_{1\leq i\leq N}$ are identically distributed, although not independent. Moreover,   coupled pairs
$$
\big((X_t^{i,N})_{t\in[0,T]},(X_t^i)_{t\in[0,T]}\big),\quad 1\leq i\leq N,
$$
are identically distributed. Hence
the estimates \eqref{ap01} and \eqref{ap02}, stated for $i=1$, also hold for any $1\le i\le N$.
\end{rem}
\subsection{Derivative flows}
This subsection prepares for the propagation of chaos at the derivative level studied in the next section by collecting several basic properties of directional and Malliavin derivatives. We remark that the results in this subsection can be extended to the case where the drift is non-globally Lipschitz continuous. For the sake of clarity and readability, the corresponding extensions and discussions are deferred to Appendix \ref{df}.

For $k\geq 2$ and $\mu\in{\scr  P}_k$, the tangent space at $\mu$ is given by
$$
T_{\mu,k}:=
L^k(\R^d\to \R^d;\mu)
:=\{\phi:\R^d\to \R^d \text{ is measurable with }\mu(|\phi|^k)<\infty\},
$$
which is a Banach space under the norm
$
\|\phi\|_{T_{\mu,k}}:=\mu(|\phi|^k)^{\ff  1 k}.
$

We recall the definition of the intrinsic derivative on $\scr P_k$; see \cite{W23b} for historical remarks on this derivative and its relation to other derivatives for functions of measures.

\begin{defn}\label{defn1}
Let $k\geq2$ and $f: {\scr P}_k\to \R$ be a continuous function, and let ${\rm id}$ be the identity map on $\R^d$.
\begin{enumerate}
\item[(1)]
$f$ is called intrinsically differentiable on ${\scr P}_k$ if, for any $\mu\in {\scr P}_k$,
\begin{align*}
	T_{\mu,k}\ni\phi\mapsto D^{I}_{\phi}f(\mu):=\lim_{\vv\downarrow0}\ff{f\left( \mu\circ({\rm id}+\vv\phi)^{-1}\right) -f(\mu)}{\vv}\in\R
\end{align*}
is a well-defined bounded linear functional. In this case, there exists
$$
D^If(\mu)\in T_{\mu,k}^*:=L^{k^*}(\R^d\to\R^d;\mu),\ \ k^*:=\ff k{k-1},
$$
such that
$$
\int_{\R^d}\<D^If(\mu)(x), \phi(x)\>  \mu(\d x) = D_\phi^I f(\mu),\ \ \phi\in T_{\mu,k}.
$$

\item[(2)]
$f$ is called $L$-differentiable on ${\scr P}_k$ with $D^Lf=D^If$ if it is intrinsically differentiable and
\begin{equation*}
	\lim_{\|\phi\|_{T_{\mu,k}}\to0}\ff{\left|f\left( \mu\circ({\rm id}+\phi)^{-1}\right)-f(\mu)-D_\phi^{I}f(\mu)\right|}{\|\phi\|_{T_{\mu,k}}}=0,\ \ \mu\in\scr P_k.
\end{equation*}

\item[(3)]
$\D_k$ is the class of continuous functions $g$ on $\R^d\times{\scr P}_k$, with values in $\R$, $\R^d$ or a matrix space, such that $g(x,\mu)$ is differentiable in $x$, $L$-differentiable in $\mu$, and $D^{L}g(x,\mu)(y)$ has a version jointly continuous in $(x,y,\mu)\in \R^d\times \R^d\times{\scr P}_k$ such that
$$
|D^{L}g(x,\mu)(y)|\le c(x,\mu)(1+|y|^{k-1}),\ \ x,y\in\R^d,\mu\in  {\scr P}_k,
$$
for some positive function $c$ on $\R^d\times{\scr P}_k$. For vector or matrix valued functions, the derivatives are understood componentwise, and $|\cdot|$ denotes the corresponding Euclidean or operator norm.
\end{enumerate}
\end{defn}
To ensure the existence of  the derivative flows for \eqref{E1} and \eqref{E2}   in $L^k(\OO\to C([0,T];\R^d),\P)$, we make the following assumption. Note that ${\bf(H_1)}$ below can imply ${\bf(H)}$.

\begin{enumerate}\item[$\bf(H_1)$]  
Let $k\geq2$. The coefficients $b,\si$ satisfy 
\begin{align*}
\sup_{(t,x,\mu)\in[0,T]\times\R^d\times\scr P_k}\left\{|b_t({\bf 0},\dd_{{\bf 0}})|+\|\nn b_t(x,\mu)\|+\|\si_t({\bf 0},\dd_{{\bf 0}})\|+\|\nn \si_t(x,\mu)\|\right\}<\8,
\end{align*}
and for any $t\in[0,T]$, $b_t,\sigma_t\in\D_k$ such that
\begin{align}\label{l3}
\sup_{(t,x,y,\mu)\in[0,T]\times\R^d\times\R^d\times\scr P_k}\left\{\|D^Lb_t(x,\mu)(y)\|+\|D^L\si_t(x,\mu)(y)\|\right\}<\8.
\end{align}
\end{enumerate}

As the coefficients of the equation \eqref{E2} depend on the empirical measure,  we first recall the following differentiation formula with respect to the empirical measure, see \cite[Proposition 5.35]{CD}, which will be used to investigate the derivative flows associated with \eqref{E2}. Note that in that reference, the result is only stated for $k=2$ since the Lions derivative is defined in $\scr P_2$. Nevertheless, the same proof applies to the present $L$-differentiability setting on $\scr P_k$ for $k\ge2$. To avoid repetition, we omit the proof.
\begin{lem}\label{de}
Let $k\geq2$, if $f:\scr P_k\to\R$ is continuously $L$-differentiable, then
\begin{align*}
\nn_{x^i}f\left( \ff 1 N\sum_{j=1}^N\delta_{x^j}\right) =\ff 1 ND^Lf\left( \ff 1 N\sum_{j=1}^N\delta_{x^j}\right) (x^i),\ \ x^i\in\R^d,\,1\le i\le N.
\end{align*}
\end{lem}

\subsubsection{Directional Derivative Flows}
Let $k\geq2$ and $X_0,\eta\in L^k(\OO\to\R^d,\F_0,\P)$. As before, let $(X_0^i,\eta^i,W^i)_{1\leq i\leq N}$ be i.i.d. copies of $(X_0,\eta,W)$, and set $\boldsymbol\eta=(\eta^1,\eta^2,\cdots,\eta^N)$. For any $\vv\geq0$ and
$1\leq i\leq N$, let $X_t^{i,\vv}$ and $X_t^{i,N,\vv}$ solve the SDEs
\begin{align}\label{E10}
\d X_t^{i,\vv}
=
b_t(X_t^{i,\vv},\mu_t^\vv)\d t
+\si_t(X_t^{i,\vv},\mu_t^\vv)\d W_t^i,\quad t\in[0,T],
\end{align}
and
\begin{align}\label{E20}
\d X_t^{i,N,\vv}
=
b_t(X_t^{i,N,\vv},\hat\mu_t^{N,\vv})\d t
+\si_t(X_t^{i,N,\vv},\hat\mu_t^{N,\vv})\d W_t^i,\quad t\in[0,T],
\end{align}
respectively, with
$$
X_0^{i,\vv}=X_0^{i,N,\vv}=X_0^i+\vv\eta^i,
$$ where
\begin{align}\label{ev}
\mu_t^\vv:=\L_{X_t^{i,\vv}},\quad
\hat\mu_t^{N,\vv}:=\ff 1 N\sum_{j=1}^N\dd_{X_t^{j,N,\vv}}.
\end{align}
The definition of $\mu_t^\vv$ is independent of $i$, since $(X_0^i,\eta^i,W^i)_{1\leq i\leq N}$ are i.i.d. and \eqref{E10} is well-posed. In particular, when $\vv=0$, $\mu_t^\vv=\mu_t$, $X_t^{i,\vv}=X_t^i$, and $X_t^{i,N,\vv}=X_t^{i,N}$. 

The following lemma recalls that
\begin{align}\label{vi1}
V_t^i:=\nn_{\eta^i}X_t^i
:=\lim_{\vv\downarrow0}\ff{X_t^{i,\vv}-X_t^i}{\vv},
\quad t\in[0,T],\,1\leq i\leq N,
\end{align}
and
\begin{align}\label{vi2}
V_t^{i,N}:=\nn_{\boldsymbol\eta}X_t^{i,N}
:=\lim_{\vv\downarrow0}\ff{X_t^{i,N,\vv}-X_t^{i,N}}{\vv},
\quad t\in[0,T],\,1\leq i\leq N,
\end{align}
exist in $L^k(\OO\to C([0,T];\R^d),\P)$. 
\begin{lem}\label{gf1} Assume $\bf(H_1)$. Then, for any $X_0,\eta\in L^k(\OO\to\R^d,\F_0,\P)$, the processes
$(V_t^i)_{t\in[0,T]}$ and $(V_t^{i,N})_{t\in[0,T]}$ exist in
$L^k(\OO\to C([0,T];\R^d),\P)$, and uniquely solve
\begin{equation}\label{v1}
\left\{
\begin{aligned}
	\d V_t^i
	=&\left\{\nn_{V_t^i}b_t(X_t^i,\mu_t)
	+\left(\E\Big\<D^Lb_t(y,\cdot)(\mu_t)(X_t^i),V_t^i\Big\>\Big|_{y=X_t^i}\right)\right\}\d t\\
	&+\left\{\nn_{V_t^i}\si_t(X_t^i,\mu_t)
	+\left(\E\Big\<D^L\si_t(y,\cdot)(\mu_t)(X_t^i),V_t^i\Big\>\Big|_{y=X_t^i}\right)\right\}\d W_t^i,\\
	&V_0^i=\eta^i,\quad 1\leq i\leq N,\ t\in[0,T],
\end{aligned}
\right.
\end{equation}
and
\begin{equation}\label{v2}
\left\{
\begin{aligned}
	\d V_t^{i,N}
	=&\left\{\nn_{V_t^{i,N}}b_t(X_t^{i,N},\hat\mu_t^N)
	+\ff 1 N\sum_{j=1}^N
	\Big\<D^Lb_t(X_t^{i,N},\cdot)(\hat\mu_t^N)(X_t^{j,N}),V_t^{j,N}\Big\>\right\}\d t\\
	&+\left\{\nn_{V_t^{i,N}}\si_t(X_t^{i,N},\hat\mu_t^N)
	+\ff 1 N\sum_{j=1}^N
	\Big\<D^L\si_t(X_t^{i,N},\cdot)(\hat\mu_t^N)(X_t^{j,N}),V_t^{j,N}\Big\>\right\}\d W_t^i,\\
	&V_0^{i,N}=\eta^i,\quad 1\leq i\leq N,\ t\in[0,T],
\end{aligned}
\right.
\end{equation} respectively. 
Moreover,  for any  $m\geq1$, 
\begin{align}\label{ve1}
\E\left[\sup_{t\in[0,T]}|V_t^i|^m\Big|\F_0\right]\leq c(1+|\eta^i|^m),
\quad 1\leq i\leq N,
\end{align}
and for any $m\geq k$,
\begin{align}\label{ve2}
\E\left[\sup_{t\in[0,T]}|V_t^{i,N}|^m\Big|\F_0\right]
\leq c\left(|\eta^i|^m+\ff 1 N\sum_{j=1}^N|\eta^j|^m\right),
\quad 1\leq i\leq N,
\end{align}
hold for some constant $c>0$.
\end{lem}
\begin{proof}
The existence of \eqref{vi1}, its satisfaction of \eqref{v1}, and the estimate
\eqref{ve1} follow from
\cite[Lemma 4.1, Proposition 4.2]{W23b}. For the particle
system, as in Appendix~\ref{2.1}, we regard \eqref{E2} as a classical SDE on
$\R^{dN}$.  Applying \cite[Theorem 2.1, Proposition 4.2]{W23b} to this $\R^{dN}$-valued
system, and using Lemma~\ref{de} to identify the derivatives of the empirical-measure
terms, we obtain the existence of \eqref{vi2} and its equation \eqref{v2}.
Therefore, it remains to prove \eqref{ve2}.

Since $m\geq k\geq2$,  using \eqref{v2}, H\"older's and Jensen's inequalities, and assumption $\bf(H_1)$, there exists a constant $c_1(m)>0$ such that for any $t\in[0,T]$ and $1\leq i\leq N$,
\begin{align}\label{v01}
\d |V_t^{i,N}|^m
\leq c_1(m)\left\{|V_t^{i,N}|^m+\ff 1 N\sum_{j=1}^N|V_t^{j,N}|^m\right\}\d t+\d M_t^i,
\end{align}
where
\begin{align*}
\d \<M^i\>_t
\leq c_1(m)|V_t^{i,N}|^{2(m-1)}
\left(|V_t^{i,N}|^2+\ff 1 N\sum_{j=1}^N|V_t^{j,N}|^2\right)\d t.
\end{align*}
Here we used
$$
\left(\ff1N\sum_{j=1}^N|V_t^{j,N}|\right)^m
\leq \ff1N\sum_{j=1}^N|V_t^{j,N}|^m,
$$
and the same type of estimate for the diffusion term.
Applying the Burkholder--Davis--Gundy, Jensen's and Young's inequalities, for any $t\in[0,T]$, we obtain
\begin{align}\label{v02}
\begin{split}
	\E\left[\sup_{s\in[0,t]}|V_s^{i,N}|^m\bigg|\F_0\right]
	\leq& |\eta^i|^m
	+\ff12\E\left[\sup_{s\in[0,t]}|V_s^{i,N}|^m\bigg|\F_0\right]\\
	&+c_2(m)\int_0^t
	\E\left[|V_s^{i,N}|^m+\ff1N\sum_{j=1}^N|V_s^{j,N}|^m\bigg|\F_0\right]\,\d s
\end{split}
\end{align}
for some constant $c_2(m)>0$.

Averaging \eqref{v02} over $i=1,2,\cdots,N$, for any $t\in[0,T]$, we derive
\begin{align*}
\ff 1 N\sum_{i=1}^N
\E\left[\sup_{s\in[0,t]}|V_s^{i,N}|^m\bigg|\F_0\right]
\leq
\ff 2 N\sum_{i=1}^N|\eta^i|^m
+4c_2(m)\int_0^t
\ff 1 N\sum_{i=1}^N
\E\left[\sup_{r\in[0,s]}|V_r^{i,N}|^m\bigg|\F_0\right]\d s,
\end{align*}
then Gr\"onwall's inequality yields that 
\begin{equation*}
\begin{aligned}
	\ff 1 N\sum_{i=1}^N\E\left[\sup_{s\in[0,t]}|V_s^{i,N}|^m\bigg|\F_0\right]
	\leq \ff {2\e^{4c_2(m)T}} N\sum_{i=1}^N|\eta^i|^m, \ \ t\in[0,T].
\end{aligned}
\end{equation*}
Plugging this into \eqref{v02} and using Gr\"onwall's inequality once more, we obtain \eqref{ve2}.
\end{proof}
\subsubsection{Malliavin derivative flows}
Consider the Cameron--Martin space
$$\H= \bigg\{h\in C([0,T]\to \R^m): h_0={\bf0}, h_t'\ \text{exists\ a.e.}\ t, \|h\|_\H^2:= \int_0^T | h_t'|^2\d t<\infty\bigg\}.$$
Let $\mu_T$ be the distribution of $W_{[0,T]}:=\{W_t\}_{t\in [0,T]}$, which is the Wiener measure on the path space
$\C_T:=C([0,T]\to\R^m)$. For $F\in L^2(\R^d\times \C_T,\mu\times \mu_T)$, $F(X_0,W_{[0,T]})$ is called Malliavin differentiable along a deterministic direction $h\in\H$ if the directional derivative
$$
D_hF(X_0,W_{[0,T]}):=
\lim_{\vv\to0}\ff{F(X_0,W_{[0,T]}+\vv h)-F(X_0,W_{[0,T]})}{\vv}
$$
exists in $L^2(\OO,\P)$. If the map $\H\ni h\mapsto D_hF\in L^2(\OO,\P)$ is bounded, then there exists a unique
$DF(X_0,W_{[0,T]})\in L^2(\OO\to\H,\P)$ such that
$$
\<DF(X_0,W_{[0,T]}),h\>_\H=D_hF(X_0,W_{[0,T]}),\quad h\in\H.
$$
In this case, we write $F(X_0,W_{[0,T]})\in\D(D)$ and call $DF(X_0,W_{[0,T]})$ the Malliavin derivative of $F(X_0,W_{[0,T]})$. 
When $h\in L^2(\OO\to\H,\P)$ is random, we define
$$
D_hF:=\<DF,h\>_\H,\quad F\in\D(D),
$$
whenever the right-hand side is well defined.
It is well known that $(D,\D(D))$ is a closed linear operator from
$L^2(\OO,\F_T,\P)$ to $L^2(\OO\to\H,\F_T,\P)$. The adjoint operator
$(D^*,\D(D^*))$ of $(D,\D(D))$ is called the Malliavin divergence. For simplicity, in the sequel we denote $F(X_0,W_{[0,T]})$ by $F$. Then we have the integration by parts formula
\beq\label{INT}
\E\big(D_hF\big|\F_0\big)=\E\big(FD^*(h)\big|\F_0\big),\quad F\in\D(D),\ h\in\D(D^*).
\end{equation}
Moreover, for adapted $h\in L^2(\OO\to\H,\P)$, one has $h\in\D(D^*)$ with
\beq\label{DD*}
D^*(h)=\int_0^T \<h_t',\d W_t\>.
\end{equation}
For more details and applications of Malliavin calculus, one may refer to \cite{NN} and the references therein.

To calculate the Malliavin derivatives of $X_t^i$ and $X_t^{i,N}$ with
$\L_{X_0}=\mu\in\scr P_k$, $k\geq2$, we write $X_t^i=F_t^i(X_0^i,W^i_\cdot)$
and $X_t^{i,N}=F_t^{i,N}(\boldsymbol X_0,\boldsymbol W_\cdot)$ as functionals of
the initial values and the Brownian motions, where
$$
\boldsymbol X_0=(X_0^1,\cdots,X_0^N),\quad
\boldsymbol W_t=(W_t^1,\cdots,W_t^N).
$$
Then,  for adapted
$\boldsymbol h=(h^1,\cdots,h^N)$ with $h^i\in L^k(\OO\to\H,\P)$, 
\begin{align*}
D_{h^i}X_t^i
&=\lim_{\vv\downarrow0}
\ff{F_t^i(X_0^i,W^i_\cdot+\vv h^i_\cdot)-F_t^i(X_0^i,W^i_\cdot)}{\vv},\\
D_{\boldsymbol h}X_t^{i,N}
&=\lim_{\vv\downarrow0}
\ff{F_t^{i,N}(\boldsymbol X_0,\boldsymbol W_\cdot+\vv\boldsymbol h_\cdot)
-F_t^{i,N}(\boldsymbol X_0,\boldsymbol W_\cdot)}{\vv}.
\end{align*}

We shall study Malliavin derivatives along adapted random directions, which requires the diffusion coefficient to be distribution independent. Hence, in the rest of this part, we consider \eqref{E1} and \eqref{E2} with $\si_t(x,\mu)=\si_t(x)$ and make the following assumption.
\begin{enumerate}
\item[$\bf(H_2)$] $\si_t(x,\mu)=\si_t(x)$, and $a:=\si\si^*$ is invertible with $\|a^{-1}\|_\8<\8$, where $\si^*$ is the transposition of $\si$.
\end{enumerate}
For any $0\leq r<T$ and $g\in C^1([r,T])$ with $g_r=0$ and $g_T=1$, define
\begin{equation}\label{LLP1}
\begin{split}
h_t^{i}
&:=\int_0^t 1_{\{s\geq r\}}(\si_s^*a_s^{-1})(X_s^{i})g_s' V_s^{i}\,\d s,\\
h_t^{i,N}
&:=\int_0^t 1_{\{s\geq r\}}(\si_s^*a_s^{-1})(X_s^{i,N})g_s' V_s^{i,N}\,\d s,
\quad t\in[0,T],\ 1\leq i\leq N,
\end{split}
\end{equation}
where $V_s^i$ and $V_s^{i,N}$ solve \eqref{v1} and \eqref{v2}, respectively. By $\bf(H_2)$ and the estimates in Lemma \ref{gf1}, $h^i,h^{i,N}\in L^k(\OO\to\H,\P)$ are adapted. Set
$$
\boldsymbol h=(h^1,h^2,\cdots,h^N),\quad
\boldsymbol h^N=(h^{1,N},h^{2,N},\cdots,h^{N,N}).
$$
Since $X_t^i$ depends only on $W^i$, we write
\begin{align}\label{LD00}
U_t^i:=D_{h^i}X_t^i=D_{\boldsymbol h}X_t^i,\quad t\in[0,T],\,1\leq i\leq N.
\end{align}
For the interacting particle system, we write
\begin{align}\label{LD0}
U_t^{i,N}:=D_{\boldsymbol h^N}X_t^{i,N},\quad t\in[0,T],\,1\leq i\leq N.
\end{align}
When different Malliavin directions are considered, we write
$U_t^i=U_t^i(h^i)$ and $U_t^{i,N}=U_t^{i,N}(\boldsymbol h^N)$ to emphasize the dependence on the direction.

By the pathwise uniqueness of \eqref{E1} and \eqref{E2}, 
the shifted processes
$$
X_t^{h,i,\vv}:=F_t^i(X_0^i,W^i_\cdot+\vv h^i_\cdot),
\quad
X_t^{h,i,N,\vv}:=F_t^{i,N}(\boldsymbol X_0,\boldsymbol W_\cdot+\vv\boldsymbol h^N_\cdot)
$$
solve the SDEs with $X_0^{h,i,\vv}=X_0^{h,i,N,\vv}=X_0^i$,
\begin{align}\label{xhi}
\d X_t^{h,i,\vv}
=&\, b_t\bigl(X_t^{h,i,\vv},\L_{X_t^i}\bigr)\d t
+\sigma_t\bigl(X_t^{h,i,\vv}\bigr)\d(W_t^i+\vv h_t^i),
\quad t\in[0,T],
\end{align}
and
\begin{align*}
\d X_t^{h,i,N,\vv}
=&\, b_t\bigl(X_t^{h,i,N,\vv},\bar\mu_t^{h,N,\vv}\bigr)\d t
+\sigma_t\bigl(X_t^{h,i,N,\vv}\bigr)\d(W_t^i+\vv h_t^{i,N}),
\quad t\in[0,T],
\end{align*}
respectively, where
$$
\bar\mu_t^{h,N,\vv}:=\ff1N\sum_{j=1}^N\delta_{X_t^{h,j,N,\vv}}.
$$
Here the law $\L_{X_t^i}$ in \eqref{xhi} is frozen, since the Malliavin
derivative is taken with respect to the Brownian motion.

The following lemma describes the Malliavin derivative flows $U_t^i$ and $U_t^{i,N}$, which are associated with the directional derivative flows $V_t^i$ and $V_t^{i,N}$.

\begin{lem}\label{P3.1} 	
Assume {$\bf(H_1)$} and {$\bf(H_2)$}. For any
$X_0,\eta\in L^k(\Omega\to\R^d,\F_0,\P)$, with
$\boldsymbol{\eta}=(\eta^1,\eta^2,\cdots,\eta^N)$, any $0\leq r<T$ and
$g\in C^1([r,T])$ with $g_r=0$ and $g_T=1$, let $h^i$ and $h^{i,N}$ be defined as in \eqref{LLP1}. Then for any $1\le i\le N$,
$(U_t^i)_{t\in[0,T]}$ and $(U_t^{i,N})_{t\in[0,T]}$
exist in $L^k(\OO\to C([0,T];\R^d),\P)$ and are the unique solutions of
\begin{equation}\label{w1}
\left\{
\begin{aligned}
\d U_t^i
&=\left\{\nn_{U_t^i}b_t(X_t^i,\mu_t)+g_t'V_t^i\right\}\d t
+\nn_{U_t^i}\si_t(X_t^i)\d W_t^i,\quad t\in[r,T],\\
U_r^i&={\bf 0},
\end{aligned}
\right.
\end{equation}
and
\begin{equation}\label{DR3}
\left\{
\begin{aligned}
\d U_t^{i,N}
=&\left\{
\nn_{U_t^{i,N}} b_t(X_t^{i,N}, \hat\mu_t^N)
+g_t' V_t^{i,N}\right.\\
&\left.\quad
+\ff 1 N\sum_{j=1}^N
\<D^Lb_t(X_t^{i,N},\hat\mu_t^N)(X_t^{j,N}),U_t^{j,N}\>
\right\}\d t\\
&\quad+\nn_{U_t^{i,N}}\si_t(X_t^{i,N})\d W_t^i,\quad t\in[r,T],\\
U_r^{i,N}&={\bf 0},
\end{aligned}
\right.
\end{equation}
respectively. Moreover,
\begin{equation}\label{LD}
U_t^{i,N}=g_tV_t^{i,N},\quad t\in [r,T].
\end{equation}
Consequently,
\begin{equation}\label{LD1}
U_T^{i,N}=V_T^{i,N}.
\end{equation}
\end{lem}
\begin{proof}
For the process $(U_t^i=D_{h^i}X_t^i)_{1\leq i\leq N,t\in[0,T]}$, the desired conclusion follows by repeating the argument of \cite[Lemma 4.2, Theorem 6.1]{BRW20}. Hence, we only prove the assertions for the particle system \eqref{LD0}.

By  $\bf(H_1)$, $\bf(H_2)$, \eqref{ve2}, it is easy to see that $h^{i,N}\in L^k(\OO\to\H,\P)$. Moreover, since $(h_t^{i,N})'=0$ for $t<r$, it follows that $$
U_t^{i,N}=D_{\boldsymbol h^N}X_t^{i,N}={\bf 0},\quad t<r,\ 1\leq i\leq N.
$$
Note that the additional drift term generated by the Malliavin derivative in the direction $h^{i,N}$ is
$$
\si_t(X_t^{i,N})(h_t^{i,N})'
=
\left\{\si_t\si_t^*a_t^{-1}\right\}(X_t^{i,N})g_t'V_t^{i,N}
=
g_t'V_t^{i,N}.
$$
Then, as in Appendix~\ref{2.1}, we regard \eqref{E2} as a classical SDE on
$\R^{dN}$. By ${\bf(H_1)}$, ${\bf(H_2)}$, Lemma~\ref{de}, and the definition of $h^{i,N}$ in \eqref{LLP1}, the classical Malliavin differentiability result for finite-dimensional SDEs, see \cite{NN}, implies that
$$
U^{i,N}\in L^k(\OO\to C([0,T];\R^d),\P),\quad 1\leq i\leq N,
$$
and that $(U_t^{i,N})_{1\leq i\leq N}$ solves \eqref{DR3}.

Moreover, since $g_T=1$, \eqref{LD1} follows immediately from \eqref{LD}. Therefore, it remains to prove  \eqref{LD}. Set
$$
\widetilde V_t^{i,N}:=g_tV_t^{i,N},\quad t\in[r,T],\ 1\leq i\leq N.
$$Due to $g_r=0$, we have $\widetilde V_r^{i,N}={\bf 0}$. By \eqref{v2} and the fact that $\si_t$ is distribution independent under ${\bf(H_2)}$,
\begin{equation}\label{v20}
\beg{split}
\d \widetilde V_t^{i,N}
=&\bigg\{
\nn_{\widetilde V_t^{i,N}} b_t(X_t^{i,N},\hat\mu_t^N)
+g_t'V_t^{i,N}\\	&\quad
+\ff 1 N\sum_{j=1}^N
\<D^Lb_t(X_t^{i,N},\hat\mu_t^N)(X_t^{j,N}),\widetilde V_t^{j,N}\>
\bigg\}\d t\\	&\quad
+\nn_{\widetilde V_t^{i,N}}\si_t(X_t^{i,N})\d W_t^i,
\quad t\in[r,T].
\end{split}
\end{equation}
Thus $(\widetilde V_t^{i,N})_{1\leq i\leq N}$ satisfies the same equation \eqref{DR3} as $(U_t^{i,N})_{1\leq i\leq N}$ with the same initial value at $t=r$. By uniqueness,
$$
U_t^{i,N}=\widetilde V_t^{i,N}=g_tV_t^{i,N},\quad t\in[r,T],\ 1\leq i\leq N,
$$ which ends the proof.
\end{proof}
For $\boldsymbol h^N=(h^{1,N},h^{2,N},\cdots,h^{N,N})$ defined by \eqref{LLP1}, set
$$
\boldsymbol h^{N,(l)}=(0,\cdots,0,h^{l,N},0,\cdots,0),\quad 1\leq l\leq N,
$$
where $h^{l,N}$ is the $l$-th component. By the linearity of the Malliavin derivative,
\begin{align}\label{dhl}
U_t^{i,N}(\boldsymbol h^{N})
=\sum_{l=1}^N	U_t^{i,N}(\boldsymbol h^{N,(l)})
=:\sum_{l=1}^N U_t^{i,N,l},\quad
t\in[0,T],\ 1\leq i\leq N.
\end{align}
Similarly, when different Malliavin directions are considered, we write
$U_t^{i,N,l}=U_t^{i,N,l}(\boldsymbol h^N):=U_t^{i,N}(\boldsymbol h^{N,(l)})$ to emphasize the dependence on the direction.

In the next section, we shall approximate $U_t^1$ by $U_t^{1,N,1}$.  Therefore, it is useful to study each component separately.
\begin{prp}\label{uil}
Assume {$\bf(H_1)$} and {$\bf(H_2)$}. Then for any
$X_0,\eta\in L^k(\Omega\to\R^d,\F_0,\P)$, any $0\leq r<T$ and
$g\in C^1([r,T])$ with $g_r=0$ and $g_T=1$, the process 
$(U_t^{i,N,l})_{t\in[0,T], 1\leq i,l\leq N}$ exists in
$L^k(\OO\to C([0,T];\R^d),\P)$ and uniquely solves
\begin{equation}\label{DR31}
\left\{
\begin{aligned}
\d U_t^{i,N,l}
=&\left\{
\nn_{U_t^{i,N,l}} b_t(X_t^{i,N},\hat\mu_t^N)
+1_{\{l=i\}}g_t'V_t^{i,N}\right.\\
&\left.\quad
+\ff 1 N\sum_{j=1}^N
\<D^Lb_t(X_t^{i,N},\hat\mu_t^N)(X_t^{j,N}),U_t^{j,N,l}\>
\right\}\d t\\
&\quad+\nn_{U_t^{i,N,l}}\si_t(X_t^{i,N})\d W_t^i,\quad t\in[r,T],\\
U_r^{i,N,l}&={\bf 0},\quad 1\leq i,l\leq N.
\end{aligned}
\right.
\end{equation}
Moreover, there exists a constant $c>0$ independent of $N,i,l$, such that
\begin{align}\label{we3}
\E\left[\sup_{t\in[0,T]}|U_t^{i,N,l}|^k\right]
\leq c\left(1_{\{l=i\}}+\ff 1 N\right)\|\eta\|_k^k,\quad 1\leq i,l\leq N,
\end{align}
where $\|\eta\|_k^k=\E|\eta|^k$.
\end{prp}
\begin{proof}
The existence and uniqueness of solutions to \eqref{DR31} are classical once  \eqref{DR31} is regarded as an SDE on $\R^{dN}$. Hence, we only prove \eqref{we3}. 
Since $(h_t^{i,N})'=0$ for $t<r$, we still have $$
U_t^{i,N,l}={\bf 0},\quad t<r,\ 1\leq i,l\leq N.
$$
Thus,	without loss of generality, we consider \eqref{DR31} with  $r=0$. 

By It\^o's formula, Young's and H\"older's inequalities, and assumption {$\bf(H_1)$}, there exists a constant $c_1>0$ such that for any $t\in[0,T]$ and $1\leq i,l\leq N$,
\begin{align*}
\d |U_t^{i,N,l}|^2
=&\d M_t^{i,l}
+2\<U_t^{i,N,l},\nn_{U_t^{i,N,l}} b_t(X_t^{i,N},\hat\mu_t^N)\>\d t
+\Big\|\nn_{U_t^{i,N,l}}\si_t(X_t^{i,N})\Big\|_{HS}^2\d t\\
&+2 \cdot 1_{\{l=i\}}g_t'\<U_t^{i,N,l},V_t^{i,N}\>\d t\\
&+\ff 2N\sum_{j=1}^N
\left\<U_t^{i,N,l},
\<D^Lb_t(X_t^{i,N},\hat\mu_t^N)(X_t^{j,N}),U_t^{j,N,l}\>
\right\>\d t\\
\leq&\d M_t^{i,l}
+c_1\left\{|U_t^{i,N,l}|^2
+|V_t^{i,N}|^2 1_{\{l=i\}}
+\left( \ff1 N\sum_{j=1}^N|U_t^{j,N,l}|\right) ^2\right\}\d t,
\end{align*}
where $$
\d\<M^{i,l}\>_t\leq c_1|U_t^{i,N,l}|^4\d t.
$$  
Using It\^o's formula once more and applying Young's inequality, there exists a constant $c_2>0$ such that
\begin{align}\label{ucomp1}
\d |U_t^{i,N,l}|^k
\leq c_2\left(|U_t^{i,N,l}|^k+1_{\{l=i\}}|V_t^{i,N}|^k+\ff1N\sum_{j=1}^N|U_t^{j,N,l}|^k\right)\d t+\d \widetilde M_t^{i,l},
\end{align}
where 
$$
\d\<\widetilde M^{i,l}\>_t\leq c_2^2|U_t^{i,N,l}|^{2k}\d t.
$$
For $t\in[0,T]$, set
$$
H_t^{i,l}:=\E\left[\sup_{s\in[0,t]}|U_s^{i,N,l}|^k\right].
$$
By  Young's and the Burkholder--Davis--Gundy inequalities, we can find a constant $c_3>0$ such that for any $1\le i,l\le N$ and $t\in[0,T]$,
\begin{equation}
\begin{aligned}\label{wm1}
H_t^{i,l}\leq & c_2\int_0^t\E\bigg[|U_s^{i,N,l}|^k+1_{\{l=i\}}|V_s^{i,N}|^k+\ff1N\sum_{j=1}^N|U_s^{j,N,l}|^k\bigg]\,\d s\\\ &+c_2\E\left(\int_0^t |U_s^{i,N,l}|^{2k}\,\d s\right)^{\ff 1 2}\\
\leq& c_3\int_0^t\bigg(H_s^{i,l}+\ff1N\sum_{j=1}^NH_s^{j,l}
+\E\left[|V_s^{i,N}|^k1_{\{l=i\}}\right]\bigg)\,\d s
+\ff12H_t^{i,l}.
\end{aligned}
\end{equation}
Taking the average over $i$ on both sides and using \eqref{ve2}, it follows 
\begin{align*}
\ff1N\sum_{i=1}^N H_t^{i,l}
&\leq 4c_3\int_0^t\ff1N\sum_{i=1}^N H_s^{i,l}\,\d s
+2c_3\int_0^t\ff1N\E|V_s^{l,N}|^k\,\d s\\
&\leq 4c_3\int_0^t\ff1N\sum_{i=1}^N H_s^{i,l}\,\d s
+\ff{c_4}{N}\|\eta\|_k^k,\quad 1\leq l\leq N,\,t\in[0,T]
\end{align*}
for some constant  $c_4>0$. Then by Gr\"onwall's inequality, 
\begin{align*}
\ff 1 N\sum_{i=1}^N	H_t^{i,l}\leq \ff {c_4\e^{4c_3T}} N\|\eta\|_k^k, \  \ 1\leq l\leq N,\,t\in[0,T].
\end{align*}
Plugging this into \eqref{wm1},  using \eqref{ve2} and Gr\"onwall's inequality once more, we can find a constant $c_5>0$ such that
\begin{align*}
H_t^{i,l}\leq c_5\left( 1_{\{l=i\}}+\ff 1 N\right)\|\eta\|_k^k, \ \ 1\leq i,l\leq N,\,t\in[0,T],
\end{align*}
which completes the proof.
\end{proof}

\section{Propagation of chaos  for derivative  flows}\label{apd}
In this section, we first approximate the directional derivative flow $V_t^1$ by $V_t^{1,N}$ for $t\in[0,T]$. Based on this approximation, we then derive estimates for the convergence of the corresponding Malliavin derivative flows $U_t^1$ and $U_t^{1,N,1}$. It is worth noting that we also obtain explicit convergence rates, which can be sharp in some cases.

In addition to {$\bf(H_1)$}, we need the following uniform continuity assumption on the first-order derivatives of the coefficients.
\begin{enumerate}
\item[$\bf(H_3)$]
$b$ and $\si$ satisfy
\begin{equation*}
\begin{aligned}
&\lim_{\vv\downarrow0}
\sup_{\substack{
	t\in[0,T],\,x,x',y,y'\in\R^d,\,\mu,\nu\in\scr P_k\\
	|x-x'|\vee \W_k(\mu,\nu)\vee |y-y'|\leq\vv}}
	\Big\{
	\|\nn b_t(x,\mu)-\nn b_t(x',\nu)\|
	+\|\nn\si_t(x,\mu)-\nn\si_t(x',\nu)\|\\
	&\quad\quad
	+\|D^Lb_t(x,\mu)(y)-D^Lb_t(x',\nu)(y')\|
	+\|D^L\si_t(x,\mu)(y)-D^L\si_t(x',\nu)(y')\|
	\Big\}=0.
\end{aligned}
\end{equation*}
\end{enumerate}
To obtain convergence rates, we impose the following quantitative version of {$\bf(H_3)$}.
\begin{enumerate}
\item[$\bf(H_4)$]
Let $q>k\geq2$. There exist constants $K>0$, $\alpha\in(0,1]$ and $0\leq \rho\leq q(k-\alpha)$ such that for any
$t\in[0,T]$, $x,x',y,y'\in\R^d$ and $\mu,\nu\in\scr P_q$,
\begin{equation}\label{l4}
\begin{aligned}
&\|\nn b_t(x,\mu)-\nn b_t(x',\nu)\|
+\|\nn\si_t(x,\mu)-\nn\si_t(x',\nu)\|\\
&\leq K\left(1+|x|^{\ff \rho k}+|x'|^{\ff \rho k}
+\|\mu\|_q^{\ff \rho k}+\|\nu\|_q^{\ff \rho k}\right)
\left(|x-x'|+\W_k(\mu,\nu)\right)^\alpha,
\end{aligned}
\end{equation}
and
\begin{equation}\label{l5}
\begin{aligned}
&\|D^Lb_t(x,\mu)(y)-D^Lb_t(x',\nu)(y')\|
+\|D^L\si_t(x,\mu)(y)-D^L\si_t(x',\nu)(y')\|\\
&\leq K\left(1+|x|^{\ff \rho k}+|x'|^{\ff \rho k}
+|y|^{\ff \rho k}+|y'|^{\ff \rho k}
+\|\mu\|_q^{\ff \rho k}+\|\nu\|_q^{\ff \rho k}\right)\\
&\quad\times
\left(|x-x'|+|y-y'|+\W_k(\mu,\nu)\right)^\alpha.
\end{aligned}
\end{equation}
Here and in what follows, for $\mu\in\scr P_q$, we write
$\|\mu\|_q:=\mu(|\cdot|^q)^{\ff1q}$.
\end{enumerate}
We first establish the propagation of chaos for the directional derivative flows.
\begin{thm}\label{gf2}
Assume {$\bf(H_1)$} and {$\bf(H_3)$}. Then the following assertions hold.
\begin{enumerate}
\item[(1)] For any $X_0,\eta\in L^k(\Omega\to\R^{d},\F_0,\P)$,
\begin{align}\label{ap1}
\lim_{N\to\infty}
\E\left[\sup_{t\in[0,T]}\big|V_t^{1,N}-V_t^1\big|^k\right]=0.
\end{align}
\item[(2)] If further assume ${\bf (H_4)}$, then for any
$X_0,\eta\in L^{q}(\Omega\to\R^{d},\F_0,\P)$, there exists a constant $c>0$  independent of $N$, such that
\begin{align}\label{ap2}
\E\left[\sup_{t\in[0,T]}\big|V_t^{1,N}-V_t^1\big|^k\right]
\leq c\epsilon(N)^{\alpha\left( \ff{q-k}{\rho+\alpha q}\wedge 1\right)},
\end{align}
where $\epsilon(N)$ is defined in \eqref{en}.
\end{enumerate}
\end{thm}
We then transfer the preceding convergence to the Malliavin derivative flows.
\begin{thm}\label{gf3}
Assume {$\bf(H_1)$}, {$\bf(H_2)$} and {$\bf(H_3)$}. Then, for any $0\leq r<T$ and
$g\in C^1([r,T])$ with $g_r=0$ and $g_T=1$, the following assertions hold.
\begin{enumerate}
\item[(1)] For any $X_0,\eta\in L^k(\Omega\to\R^{d},\F_0,\P)$,
\begin{align}\label{ap3}
\lim_{N\to\infty}
\E\left[\sup_{t\in[0,T]}\big|U_t^{1,N,1}-U_t^1\big|^k\right]=0.
\end{align}

\item[(2)] If further assume ${\bf (H_4)}$, then for any
$X_0,\eta\in L^{q}(\Omega\to\R^{d},\F_0,\P)$, there exists a constant $c>0$ independent of $N$, such that
\begin{align}\label{ap4}
\E\left[\sup_{t\in[0,T]}\big|U_t^{1,N,1}-U_t^1\big|^k\right]
\leq c\epsilon(N)^{\alpha\left( \ff{q-k}{\rho+\alpha q}\wedge 1\right)}.
\end{align}
\end{enumerate}
\end{thm}
\begin{rem}	
Similar to  Remark~\ref{exc}, these estimates are independent of the choice of the particle label. Hence the estimates in \eqref{ap1}--\eqref{ap4}, stated for $i=1$, also hold with $1$ replaced by any $1\leq i\leq N$. 
\end{rem}
\begin{rem}\label{sharp}
\begin{enumerate}
\item[(1)]
Assumption ${\bf(H_4)}$ does not require the first-order derivatives of the coefficients
to be globally $\alpha$-H\"older continuous with a uniform H\"older constant.
Instead, it allows the corresponding H\"older constants to grow polynomially
in the spatial variables. This is the role of the parameter $\rho$.
\item[(2)]
For any $k\geq2$, the above convergence rate we obtained  improves as $\alpha$ and $q$ increase and as $\rho$ decreases.

In particular, for any fixed $q>k\geq2$, the optimal rate in our framework is obtained when $\rho=0,\,\alpha=1$, that is, there exists a constant $\tt K>0$ such that 
\begin{align*}
&\|\nn b_t(x,\mu)-\nn b_t(x',\nu)\|+\|\nn\si_t(x,\mu)-\nn\si_t(x',\nu)\|\\
&+\|D^Lb_t(x,\mu)(y)-D^Lb_t(x',\nu)(y')\|
+\|D^L\si_t(x,\mu)(y)-D^L\si_t(x',\nu)(y')\|\\
\leq &\tt K \left\{|x-x'|+|y-y'|+\W_k(\mu,\nu)\right\},\ \ t\in[0,T],\,x,x',y,y'\in\R^d,\,\mu,\nu\in\scr P_q,
\end{align*}
in which case the rate becomes $\epsilon(N)^{1-\ff{k}{q}}$.
If moreover, choosing $q=\8$, that is $X_0,\eta\in L^{\8}(\Omega\to\R^{d},\F_0,\P)$,  the estimate reduces to the sharp rate 
\begin{align*}\epsilon(N)=\begin{cases}
N^{- 1 /2}, &\textrm{if } k>d/2,\\
N^{- 1 /2}  \log (1+N), & \textrm{if } k=d/2,\\
N^{-k/{d}}, & \textrm{if } k\in(0,d/2),\end{cases}
\end{align*}
matching the standard rate for the original particle system in \eqref{ap02}.
\end{enumerate}
\end{rem}
To illustrate the above results, we present below an example of $b$ satisfying {$\bf(H_1)$}, {$\bf(H_3)$} and ${\bf (H_4)}$. The same construction also applies to   $\sigma$.
\begin{exa}\label{e1}
Let
$$
b_t(x,\mu)= \int_{\R^d}\mathfrak b_t(x,y)\,\mu(\d y),
\quad (t,x,\mu)\in[0,T]\times\R^d\times\scr P_k,
$$
where $\mathfrak b: [0,T]\times \R^d\times\R^d\to \R^d$ is measurable satisfying 
$$\sup_{t\in[0,T],\,x,y\in\R^d}
\left\{
|\mathfrak b_t({\bf0},{\bf0})|
+\|\nn_x \mathfrak b_t (x,y)\|
+\|\nn_y \mathfrak b_t (x,y)\|
\right\}<\infty,$$
and  that  $\nn_x \mathfrak b_t (x,y)$, $\nn_y \mathfrak b_t (x,y)$ are uniformly continuous in $(x,y)\in\R^d\times\R^d$   uniformly in $t\in[0,T]$.
Then $b$ satisfies   {$\bf(H_1)$}, {$\bf(H_3)$}. If moreover 
\begin{equation}\label{ex4}
\begin{aligned} 
&\|\nn_x \mathfrak b_t(x,y)-\nn_x \mathfrak b_t(x',y')\|
+\|\nn_y \mathfrak b_t(x,y)-\nn_y\mathfrak b_t(x',y')\|\\
&\leq K\left(1+|x|^{\ff\rho k}+|x'|^{\ff\rho k}
+|y|^{\ff\rho k}+|y'|^{\ff\rho k}\right)
\left(|x-x'|+|y-y'|\right)^\alpha,
\end{aligned}
\end{equation}	holds for   some $\alpha\in(0,1]$, $0\leq \rho\leq q(k-\alpha)$ and for any $t\in[0,T],\,x,x',y,y'\in\R^d$,
then  $b$ satisfies ${\bf(H_4)}$. \end{exa}

\begin{proof} 		
Since for any  $(t,x,y,\mu)\in[0,T]\times\R^d\times\R^d\times\scr P_k$,
\begin{align}\label{ex1}
\nn b_t(x,\mu)=\int_{\R^d}\nn_x \mathfrak b_t(x,y)\,\mu(\d y),
\end{align}
and
\begin{align}\label{ex2}
D^Lb_t(x,\mu)(y)=\nn_y \mathfrak b_t(x,y),
\end{align}
$b$ satisfies   ${\bf(H_1)}$.

Next,  for any $  t\in[0,T],\,x,x'\in\R^d$, and $\mu,\nu\in\scr P_k$,
\begin{equation}\label{ex3}
\begin{aligned}
&\left\|\int_{\R^d}\nn_x \mathfrak b_t(x,y)\,\mu(\d y)
-\int_{\R^d}\nn_x \mathfrak b_t(x',y)\,\nu(\d y)\right\|\\
&\leq \inf_{\pi\in\C(\mu,\nu)}\int_{\R^d\times\R^d}
\|\nn_x \mathfrak b_t(x,y)-\nn_x \mathfrak b_t(x',z)\|\,\pi(\d y,\d z).
\end{aligned}
\end{equation}
Since $\nn_x \mathfrak b$ is bounded and uniformly continuous, the right-hand side converges to $0$ uniformly whenever $|x-x'|+\W_k(\mu,\nu)\to0$. The same argument can apply to $D^Lb_t(x,\mu)(y)=\nn_y\mathfrak b_t(x,y)$. Hence ${\bf(H_3)}$ holds for $b$.

Moreover,  let $\rho>0$.
For any ${\pi\in\C(\mu,\nu)}$, H\"older's and Jensen's inequalities give
\begin{align*}
&\int_{\R^d\times\R^d}
\left(1+|x|^{\ff\rho k}+|x'|^{\ff\rho k}+|y|^{\ff\rho k}+|z|^{\ff\rho k}\right)
\left(|x-x'|+|y-z|\right)^\alpha\,\pi(\d y,\d z)\\
&\leq\left( \int_{\R^d\times\R^d}
\left(1+|x|^{\ff\rho k}+|x'|^{\ff\rho k}+|y|^{\ff\rho k}+|z|^{\ff\rho k}\right)^{\ff k{k-\alpha}}\,\pi(\d y,\d z)\right) ^{\ff {k-\alpha}k}
\\&\quad\times\left( \int_{\R^d\times\R^d}
\left(|x-x'|+|y-z|\right)^k\,\pi(\d y,\d z)\right) ^{\ff \alpha k}\\
&\leq\left( \int_{\R^d\times\R^d}
\left(1+|x|^{\ff\rho k}+|x'|^{\ff\rho k}+|y|^{\ff\rho k}+|z|^{\ff\rho k}\right)^{\ff {kq}{\rho}}\,\pi(\d y,\d z)\right) ^{\ff \rho{kq}}
\\&\quad\times\left( \int_{\R^d\times\R^d}
\left(|x-x'|+|y-z|\right)^k\,\pi(\d y,\d z)\right) ^{\ff \alpha k}\\
&\leq
C\left(1+|x|^{\ff\rho k}+|x'|^{\ff\rho k}+\|\mu\|_q^{\ff\rho k}+\|\nu\|_q^{\ff\rho k}\right)
\left(|x-x'|+\left( \int_{\R^d\times\R^d}
|y-z|^k\,\pi(\d y,\d z)\right) ^{\ff 1 k}\right)^\alpha,
\end{align*}
for some constant $C>0$,	where we used $\rho\leq q(k-\alpha)$. Taking the infimum over $\pi\in\C(\mu,\nu)$, and combining with \eqref{ex4}--\eqref{ex3} yields ${\bf(H_4)}$ for $b$. 

Finally, when $\rho=0$,  ${\bf(H_4)}$ for $b$ follows directly from  \eqref{ex4} and \eqref{ex3}. This  corresponds to the usual globally $\alpha$-H\"older continuous case.
\end{proof}

Recall that when different Malliavin directions are considered,  for any $t\in[0,T]$, we denote 
$$U_t^1(h^1)=D_{h^1}X_t^1,\qquad
U_t^{1,N}(\boldsymbol h^N)=D_{\boldsymbol h^N}X_t^{1,N},\qquad
U_t^{1,N,l}(\boldsymbol h^N)
=D_{\boldsymbol h^{N,(l)}}X_t^{1,N},\,1\le l\le N$$
to emphasize the dependence on the direction. In the following, we record a corollary to  identify the limit of the Malliavin perturbations generated by the interaction terms.

\begin{cor}\label{cor1}
Assume {$\bf(H_1)$}, {$\bf(H_2)$} and {$\bf(H_3)$}. Under the notation of Theorem \ref{gf3}, let $h^i$, $h^{i,N}$ be defined as in \eqref{LLP1} and  define
\begin{equation}\label{hath}
\hat h_t^1
:=
\int_0^t1_{\{s\geq r\}}(\si_s^*a_s^{-1})(X_s^1)
\left( \E\<D^Lb_s(y,\cdot)(\L_{X_s^1})(X_s^1),g_sV_s^1\>\right)|_{y=X_s^1}\,
\d s,
\end{equation}
then for any $X_0,\eta\in L^k(\Omega\to\R^{d},\F_0,\P)$,
\begin{align}\label{ap5}
\lim_{N\to\infty}
\E\left[\sup_{t\in[0,T]}
\bigg|\sum_{l=2}^NU_t^{1,N,l}(\boldsymbol{h}^N)-U_t^1({\hat h^1})\bigg|^k\right]=0.
\end{align}
If, in addition, {$\bf(H_4)$} holds, then for any
$X_0,\eta\in L^q(\Omega\to\R^d,\F_0,\P)$, there exists a constant $c>0$ independent of $N$, such that
\begin{align}\label{ap6}
\E\left[\sup_{t\in[0,T]}
\bigg|\sum_{l=2}^NU_t^{1,N,l}(\boldsymbol{h}^N)-U_t^1({\hat h^1})\bigg|^k\right]
\leq c\epsilon(N)^{\alpha\left( \ff{q-k}{\rho+\alpha q}\wedge 1\right)}.
\end{align}
\end{cor}
\begin{rem}
This corollary explains the reason  why the  Bismut formula for McKean--Vlasov SDEs has an extra component compared with the usual SDEs. Actually, on the particle level, this component corresponds to the perturbations generated by the interaction terms  $$
U_\cdot^{i,N,l}(\boldsymbol{h}^N)=D_{\boldsymbol h^{N,(l)}}X_\cdot^{i,N},\quad 1\leq i\ne l\leq N.
$$  Although each individual perturbation is small, their accumulated effect survives in the mean-field limit and becomes the expectation term in the limiting Bismut formula, specifically the second component in \eqref{zp} below, see also  \cite[(1.16)]{W23b}, \cite[Theorem 2.1]{RW19} and so on.
\end{rem}
\begin{proof}[Proof of Corollary \ref{cor1}]
By the proof of \cite[Theorem 4.1]{RW19}, the definition of $\hat h^1$ in \eqref{hath}, and the choice of $h^1$ in \eqref{LLP1}, we have
\begin{align*}
U_t^1(h^1+\hat h^1)=D_{h^1+\hat h^1}X_t^1=g_tV_t^1,\quad t\in[r,T].
\end{align*}
On the other hand, by Lemma \ref{P3.1},
\begin{align*}
U_t^{1,N}(\boldsymbol{h}^N)=g_tV_t^{1,N},\quad t\in[r,T].
\end{align*}Since the Malliavin derivative is linear in the direction,
\begin{align*}
U_t^{1,N}(\boldsymbol{h}^N)
=
\sum_{l=1}^NU_t^{1,N,l}(\boldsymbol{h}^N),
\qquad
U_t^1(h^1+\hat h^1)
=U_t^1(h^1)+U_t^1(\hat h^1),\quad t\in[0,T].
\end{align*}
Consequently,
\begin{align*}
\sum_{l=2}^NU_t^{1,N,l}(\boldsymbol{h}^N)-U_t^1(\hat h^1)
&=
\left(U_t^{1,N}(\boldsymbol{h}^N)-U_t^1(h^1+\hat h^1)\right)
-\left(U_t^{1,N,1}(\boldsymbol{h}^N)-U_t^1(h^1)\right)\\
&=
g_t\left(V_t^{1,N}-V_t^1\right)
-\left(U_t^{1,N,1}(\boldsymbol{h}^N)-U_t^1(h^1)\right),\quad t\in[r,T].
\end{align*}
For $t<r$, all the Malliavin directions involved above vanish, so the left-hand side is zero. 
Then, since $g\in C^1([r,T])$, we have
\begin{align*}
&\E\left[\sup_{t\in[0,T]}
\bigg|\sum_{l=2}^NU_t^{1,N,l}(\boldsymbol{h}^N)-U_t^1(\hat h^1)\bigg|^k\right]\\
&\leq
c_1(k)\E\left[\sup_{t\in[0,T]}|V_t^{1,N}-V_t^1|^k\right]
+c_1(k)\E\left[\sup_{t\in[0,T]}|U_t^{1,N,1}(\boldsymbol{h}^N)-U_t^1(h^1)|^k\right]
\end{align*}
for some constant $c_1(k)>0$. 
Then \eqref{ap5} follows from \eqref{ap1} and \eqref{ap3}. If {$\bf(H_4)$} also holds, \eqref{ap6} follows from the rate estimates \eqref{ap2} and \eqref{ap4}, which completes the proof.
\end{proof}
\subsection{Proof of Theorem \ref{gf2}}
\begin{proof} For any $X_0,\eta\in L^k(\Omega\to\R^d,\F_0,\P)$ and $1\leq i\leq N$, set
$$
\Delta_t^{i,N}:=V_t^i-V_t^{i,N},\quad t\in[0,T].
$$
By \eqref{v1} and \eqref{v2}, we have
\begin{equation}\label{zi}
\begin{aligned}
\d \Delta_t^{i,N}
=&\bigg\{\nn_{\Delta_t^{i,N}}b_t(X_t^{i,N},\hat\mu_t^N)
+\ff1N\sum_{j=1}^N
\Big\<D^Lb_t(X_t^{i,N},\hat\mu_t^N)(X_t^{j,N}),\Delta_t^{j,N}\Big\>\\
&\qquad\qquad+\xi_t^{i,N}+\zeta_t^{i,N}\bigg\}\d t\\
&+\bigg\{\nn_{\Delta_t^{i,N}}\si_t(X_t^{i,N},\hat\mu_t^N)
+\ff1N\sum_{j=1}^N
\Big\<D^L\si_t(X_t^{i,N},\hat\mu_t^N)(X_t^{j,N}),\Delta_t^{j,N}\Big\>\\
&\qquad\qquad+\chi_t^{i,N}+\psi_t^{i,N}\bigg\}\d W_t^i,
\end{aligned}
\end{equation}
where\begin{equation}\label{coap}
\begin{aligned}
\xi_t^{i,N}:=&\nn_{V_t^i}b_t(X_t^i,\mu_t)
-\nn_{V_t^i}b_t(X_t^{i,N},\hat\mu_t^N)\\
&\quad+\ff1N\sum_{j=1}^N
\Big\<D^Lb_t(X_t^i,\mu_t)(X_t^j)
-D^Lb_t(X_t^{i,N},\hat\mu_t^N)(X_t^{j,N}),V_t^j\Big\>,\\
\zeta_t^{i,N}:=&
\left.\E\Big[\Big\<D^Lb_t(z,\mu_t)(X_t^i),V_t^i\Big\>\Big]\right|_{z=X_t^i}
-\ff1N\sum_{j=1}^N
\Big\<D^Lb_t(X_t^i,\mu_t)(X_t^j),V_t^j\Big\>,\\
\chi_t^{i,N}:=&\nn_{V_t^i}\si_t(X_t^i,\mu_t)
-\nn_{V_t^i}\si_t(X_t^{i,N},\hat\mu_t^N)\\
&\quad+\ff1N\sum_{j=1}^N
\Big\<D^L\si_t(X_t^i,\mu_t)(X_t^j)
-D^L\si_t(X_t^{i,N},\hat\mu_t^N)(X_t^{j,N}),V_t^j\Big\>,\\
\psi_t^{i,N}:=&
\left.\E\Big[\Big\<D^L\si_t(z,\mu_t)(X_t^i),V_t^i\Big\>\Big]\right|_{z=X_t^i}
-\ff1N\sum_{j=1}^N
\Big\<D^L\si_t(X_t^i,\mu_t)(X_t^j),V_t^j\Big\>.
\end{aligned}
\end{equation}
Since $k\geq2$, by It\^o's formula, H\"older's inequality and ${\bf(H_1)}$, there exists a constant $c_1>0$ such that
\begin{equation}\label{ph1}
\d|\Delta_t^{i,N}|^k
\leq c_1\left\{|\Delta_t^{i,N}|^k+\ff1N\sum_{j=1}^N|\Delta_t^{j,N}|^k+\Psi_t^{i,N}\right\}\d t+\d M_t^i,
\end{equation}
where
$$
\Psi_t^{i,N}:=|\xi_t^{i,N}|^k+|\zeta_t^{i,N}|^k+\|\chi_t^{i,N}\|^k+\|\psi_t^{i,N}\|^k,
$$
and $M^i$ is a local martingale satisfying
\begin{align*}
\d\<M^i\>_t
\leq c_1\left\{|\Delta_t^{i,N}|^{2k}
+|\Delta_t^{i,N}|^{2(k-1)}
\left[\left(\ff1N\sum_{j=1}^N|\Delta_t^{j,N}|\right)^{2}
+\|\chi_t^{i,N}+\psi_t^{i,N}\|^2\right]\right\}\d t.
\end{align*}
Then  using	the Burkholder--Davis--Gundy and Young inequalities, and  the identical distribution of $(\Delta_t^{i,N})_{1\leq i\leq N}$, we can find a constant $c_2>0$ such that for any $t\in[0,T]$,
\begin{equation}\label{ph2}
\begin{aligned}
&\E\left[\sup_{s\in[0,t]}|\Delta_s^{i,N}|^k\right]\leq c_1\int_0^t\E\left[|\Delta^{i,N}_s|^k+\ff 1 N\sum_{j=1}^N|\Delta_s^{j,N}|^k+\Psi^{i,N}_s\right]\,\d s\\
&\quad\quad\quad+c_1\E\left(\int_0^t\left\{|\Delta_s^{i,N}|^{2k}+|\Delta_s^{i,N}|^{2(k-1)}\left[\left(\ff 1 N\sum_{j=1}^N|\Delta_s^{j,N}|\right)^{2 }+\|\chi^{i,N}_s+\psi^{i,N}_s\|^2\right]\right\}\,\d s\right)^{\ff 1 2} \\
&\leq c_2\int_0^t\E\left[|\Delta_s^{i,N}|^k+\Psi^{i,N}_s\right]\,\d s+\ff 1 2\E\left[\sup_{s\in[0,t]}|\Delta_s^{i,N}|^k\right].
\end{aligned}
\end{equation}
Thus, due to \eqref{ve1}, \eqref{ve2} and  Grönwall's inequality, there exists a  constant $c_3>0$ such that for any $t\in[0,T]$, $1\le i\le N$,
\begin{align}\label{ph3}
\E\left[\sup_{s\in[0,t]}|\Delta_s^{i,N}|^k\right]\leq c_3\int_0^t\E\Psi^{i,N}_s\,\d s<\infty.
\end{align}

\item[(1)] We first prove \eqref{ap1}. By \eqref{ph3}, it is enough to prove
$
\lim_{N\to\infty}\int_0^T\E\Psi_t^{1,N}\d t=0.
$
Actually,	we only need to prove the  	\begin{align}\label{psi}\lim_{N\to\infty}\E\Psi_t^{1,N}=0,\ t\in[0,T]\end{align} due to ${\bf (H_1)}$, the uniform moment estimates in \eqref{ve1} and \eqref{ve2} and  dominated convergence theorem.

We begin with $\xi_t^{1,N}$ and $\chi_t^{1,N}$. Due to  {$\bf(H_1)$} and {$\bf(H_3)$}, for any $n\geq1$, it follows\begin{equation}\label{pn1}
\begin{aligned}
|\xi_t^{1,N}|+\|\chi_t^{1,N}\|\leq& |V_t^1|\left\{n\left( |X_t^{1}-X_t^{1,N}|+\W_k(\mu_t,\hat\mu_t^N)\right) ^{\ff 1{k}}+s_n\right\}+\ff 1 N\sum_{j=1}^N\bigg(|V_t^j|\\
&\cdot\left\{n\left( |X_t^{1}-X_t^{1,N}|+|X_t^{j}-X_t^{j,N}|+\W_k(\mu_t,\hat\mu_t^N)\right) ^{\ff 1{k}}+s_n\right\}\bigg),
\end{aligned}
\end{equation}
where for 
\begin{align*}
\varphi(r):=
&\sup_{\substack{
	t\in[0,T],\,x,x',y,y'\in\R^d,\,\mu,\nu\in\scr P_k\\
	|x-x'|\vee \W_k(\mu,\nu)\vee |y-y'|\leq r}}
	\bigg\{
	\|\nn b_t(x,\mu)-\nn b_t(x',\nu)\|
	+\|\nn\si_t(x,\mu)-\nn\si_t(x',\nu)\|\\
	&\quad\quad+\|D^Lb_t(x,\mu)(y)-D^Lb_t(x',\nu)(y')\|+\|D^L\si_t(x,\mu)(y)-D^L\si_t(x',\nu)(y')\|
	\bigg\},
\end{align*}
we have \begin{align}\label{sn}
s_n:=\sup_{r\geq0}\left\{\varphi(r)-nr^{\ff 1{k}}\right\}	\downarrow0 \ \text{as}\ n\uparrow\8.
\end{align}
For any $1\leq j\leq N$ and $t\in[0,T]$,  let $$H_t^{j,N,n}=n^k\left( |X_t^{1}-X_t^{1,N}|+|X_t^{j}-X_t^{j,N}|+\W_k(\mu_t,\hat\mu_t^N)\right) +s_n^k,\ \ n\geq1,$$ 
which by \eqref{ap01} and \eqref{sn}, satisfies 
\begin{align}\label{hn}
\lim_{n\to\8}\lim_{N\to\8}\max_{1\leq j\leq N}\E H_t^{j,N,n}=0,\ \ t\in[0,T].
\end{align}

On the other hand, due to ${\bf(H_1)}$ and H\"older's inequality, it follows 
\begin{equation*}
\begin{aligned}
&\E\left[|\xi_t^{1,N}|^k+\|\chi_t^{1,N}\|^k\right]\leq \ff {c_5} N\sum_{j=1}^N\E\left[\left( |V_t^1|^k+|V_t^j|^k\right)\big(H_t^{j,N,n}\wedge c_4\big) \right]
\end{aligned}
\end{equation*}
for some   constants $c_4,c_5>0$. Since by \eqref{ve1} and \eqref{hn},
\begin{align*}
&	\sup_{1\leq j\leq N}\E\left[\left( |V_t^1|^k+|V_t^j|^k\right)\big(H_t^{j,N,n}\wedge c_4\big) \right]
\\\leq &2c_4\E\left[ |V_t^1|^k\1_{\{|V_t^1|^k>M\}} \right]+2M\sup_{1\leq j\leq N} \E H_t^{j,N,n},\ \ M\geq1
\end{align*}
converges to $0$ in the successive limits $N\to\infty$, $n\to\8$, and then $M\to\infty$, 
we have 
\begin{align}\label{psi1}
\lim_{N\to\8}\E\left[|\xi_t^{1,N}|^k+\|\chi_t^{1,N}\|^k\right]=0.
\end{align}
Subsequently, since $(X_t^i,V_t^i)_{1\leq i\leq N}$ are i.i.d., according to \eqref{l3}, \eqref{ve1},  we have
\begin{equation}\label{psi2}
\begin{split}
\E|\zeta_t^{1,N}|^k\leq& c_6\E\left|\E\left[\<D^Lb_t(z,\mu_t)(X_t^1),V_t^1\>\right]|_{z=X_t^1}-\ff 1 N\sum_{j=2}^N\<D^Lb_t(X_t^1,\mu_t)(X_t^j),V_t^j\>\right|^k\\&+c_6\E\left|\ff 1 N\<D^Lb_t(X_t^1,\mu_t)(X_t^1),V_t^1\>\right|^k\\
= &c_6\E\left\{ \E\bigg[ \bigg|\E\big[\<D^Lb_t(z,\mu_t)(X_t^1),V_t^1\>\big]-\ff 1 N\sum_{j=2}^N\<D^Lb_t(z,\mu_t)(X_t^j),V_t^j\>\bigg|^k\bigg]\bigg|_{z=X_t^1} \right\} \\&+c_6\E\left|\ff 1 N\<D^Lb_t(X_t^1,\mu_t)(X_t^1),V_t^1\>\right|^k\\
\le &c_7\E\left\{ \E\bigg[ \bigg|\E\big[\<D^Lb_t(z,\mu_t)(X_t^1),V_t^1\>\big]-\ff 1 N\sum_{j=1}^N\<D^Lb_t(z,\mu_t)(X_t^j),V_t^j\>\bigg|^k\bigg]\bigg|_{z=X_t^1} \right\} \\&+c_7\E\left\{\E\left[ \left|\ff 1 N\<D^Lb_t(z,\mu_t)(X_t^1),V_t^1\>\right|^k\right] \bigg|_{z=X_t^1}\right\}\\&+c_7\E\left|\ff 1 N\<D^Lb_t(X_t^1,\mu_t)(X_t^1),V_t^1\>\right|^k
\leq  {c_8}{N^{-\ff k2}}
\end{split}
\end{equation}
for some constants $c_6,c_7,c_8>0$.
Repeating the same argument above for $b$ replaced by $\si$, we can also find a constant $c_9>0$ such that
\begin{align}\label{psi3}
\E|\psi_t^{1,N}|^k\leq  {c_9}{N^{-\ff k2}}.
\end{align}
Therefore, the desired estimate  \eqref{psi} is derived by combining \eqref{psi1}, \eqref{psi2} with \eqref{psi3}.
\item[(2)]  Next, for any $X_0,\eta\in L^{q}(\Omega\to\R^{d},\F_0,\P)$ for some $q>k$, we focus on \eqref{ap2}.  Combining  \eqref{ph3} with \eqref{psi2} and \eqref{psi3}, we  only need to  prove \begin{align}\label{psi1'}
\E\left[|\xi_t^{1,N}|^k+\|\chi_t^{1,N}\|^k\right]\leq C_{0}\epsilon(N)^{\alpha\left( \ff{q-k}{\rho+\alpha q}\wedge 1\right)}, \  \ t\in[0,T]
\end{align}for some  constant $C_{0}>0$. 
Given {$\bf(H_1)$}, {$\bf(H_4)$}, and $\eta\in L^{q}(\Omega\to\R^{d},\F_0,\P)$ for $q>k$,  due to   \eqref{ve1}, the i.i.d. property of $(V_t^i)_{1\leq i\leq N}$,  and by H\"older's inequality, we  can find constants $C_1,C_2,C_3>0$ such that
\begin{equation}
\begin{aligned}\label{psi2'}
\E\left[|\xi_t^{1,N}|^k+\|\chi_t^{1,N}\|^k\right]\leq& \ff {C_1} N\sum_{j=1}^N\E\left[  \left(|V_t^1|^k+|V_t^j|^k\right)(J_t^{j,N}) ^{k}\right]  \\
\leq &C_2\left(\E|V_t^{1}|^{q}\right)^{\ff k {q}}\cdot\ff {1} N\sum_{j=1}^N\left(\E\left[(J_t^{j,N})^{\ff {kq}{q-k}}\right]\right)^{\ff {q-k} {q}}\\
\leq&\ff {C_3\|\eta\|_{{q}}^k} N\sum_{j=1}^N\left(\E\left[(J_t^{j,N})^{\ff {kq}{q-k}}\right]\right)^{\ff  {q-k}{q}},\ \ t\in[0,T],
\end{aligned}
\end{equation}
where \begin{align*}
J_t^{j,N}:=& \|\nn b_t(X_t^1,\mu_t)-\nn b_t(X_t^{1,N},\hat\mu_t^N)\|+\|D^Lb_t(X_t^1,\mu_t)(X_t^j)-D^Lb_t(X_t^{1,N},\hat\mu_t^N)(X_t^{j,N})\|\\
&+\|\nn \si_t(X_t^1,\mu_t)-\nn \si_t(X_t^{1,N},\hat\mu_t^N)\|_{HS}+\|D^L\si_t(X_t^1,\mu_t)(X_t^j)-D^L\si_t(X_t^{1,N},\hat\mu_t^N)(X_t^{j,N})\|.
\end{align*}Choosing  $$\beta=\ff{q-k}{\rho+\alpha q}\wedge 1$$ such that $$\ff {\rho\beta}{q-k}+\ff {\alpha\beta q}{q-k}\leq  1,$$ and using ${(\bf H_1)}$, ${(\bf H_4)}$, H\"older's and Young's inequalities, \eqref{MM1}, \eqref{MM2} and \eqref{ap02}, we derive
\begin{equation*}
\begin{aligned}
&\E\left[(J_t^{j,N})^{\ff {kq}{q-k}}\right]\leq K \E\bigg[C_4\wedge\bigg\{ \Big(1+|X_t^1|^q+|X_t^{1,N}|^q+|X_t^j|^q+|X_t^{j,N}|^q+\ff 1 N\sum_{j=1}^N|X_t^{j,N}|^q\Big)^{\ff {\rho}{q-k}}\\
&\quad\quad\quad\quad\quad\quad\quad\times\Big(|X_t^1-X_t^{1,N}|^k+|X_t^{j}-X_t^{j,N}|^k
+\W_k(\mu_t,\hat\mu_t^N)^k\Big)^{{\ff {\alpha q}{q-k}}}\bigg\}\bigg]\\
&\leq C_5\E\bigg[ \Big(1+|X_t^1|^q+|X_t^{1,N}|^q+|X_t^j|^q+|X_t^{j,N}|^q+\ff 1 N\sum_{j=1}^N|X_t^{j,N}|^q\Big)^{\ff {\beta \rho}{q-k}}\\
&\quad\times\Big(|X_t^1-X_t^{1,N}|^k+|X_t^{j}-X_t^{j,N}|^k
+\W_k(\mu_t,\hat\mu_t^N)^k\Big){^{\ff {\alpha\beta q}{q-k}}}\bigg]\\
&\leq C_6\left( \E\left[\Big(1+|X_t^1|^q+|X_t^{1,N}|^q+|X_t^j|^q+|X_t^{j,N}|^q+\ff 1 N\sum_{j=1}^N|X_t^{j,N}|^q\Big)\right]\right)^{1-\ff {\alpha\beta q} {q-k}}\\
&\quad\times\left( \E\left[\Big(|X_t^1-X_t^{1,N}|^k+|X_t^{j}-X_t^{j,N}|^k
+\W_k(\mu_t,\hat\mu_t^N)^k\Big)\right]\right) ^{\ff {\alpha\beta q} {q-k}}\\
&\leq  C_7\epsilon(N)^{\ff {\alpha\beta q} {q-k}}, \ \ t\in[0,T],\,1\le j\leq N
\end{aligned}
\end{equation*}
for some constants $C_4,C_5,C_6,C_7>0$. Combining this  with \eqref{psi2'} and the definition of $\beta$, we obtain the \eqref{psi1'}.
\end{proof}
\subsection{Proof of Theorem~\ref{gf3}}
\begin{proof}For $1\leq i\leq N$ and $t\in[0,T]$, set
$$
\Xi_t^{i,N}:=U_t^i-U_t^{i,N,i}.
$$
Noting that $(h_t^{i,N})'=(h_t^i)'=0$ for $t<r$, we have
$$
U_t^i=U_t^{i,N,l}={\bf0},\quad t<r,\ 1\leq i,l\leq N.
$$
Thus it suffices to consider $t\in[r,T]$. Without loss of generality, we take $r=0$.
By \eqref{w1} and \eqref{DR31}, for $1\leq i\leq N$, $t\in[0,T]$,
\begin{equation}\label{xi}
\begin{aligned}
\d \Xi_t^{i,N}
=&\left\{\nn_{\Xi_t^{i,N}}b_t(X_t^{i,N},\hat\mu_t^N)
+\tilde\xi_t^{i,N}-\tilde\zeta_t^{i,N}\right\}\d t+\left\{\nn_{\Xi_t^{i,N}}\si_t(X_t^{i,N})
+\tilde\chi_t^{i,N}\right\}\d W_t^i,
\end{aligned}
\end{equation}
where
\begin{align*}
\tilde\xi_t^{i,N}
:=&\nn_{U_t^i}b_t(X_t^i,\mu_t)
-\nn_{U_t^i}b_t(X_t^{i,N},\hat\mu_t^N)
+g_t'(V_t^i-V_t^{i,N}),\\
\tilde\zeta_t^{i,N}
:=&\ff1N\sum_{j=1}^N
\Big\<D^Lb_t(X_t^{i,N},\hat\mu_t^N)(X_t^{j,N}),U_t^{j,N,i}\Big\>,\\
\tilde\chi_t^{i,N}
:=&\nn_{U_t^i}\si_t(X_t^i)-\nn_{U_t^i}\si_t(X_t^{i,N}).
\end{align*}By It\^o's formula, the Burkholder--Davis--Gundy and Young inequalities, and Gr\"onwall's inequality, there exists a constant $c_0>0$ such that
\begin{align}\label{xi-s}
\E\left[\sup_{t\in[0,T]}|\Xi_t^{1,N}|^k\right]
\leq c_0\int_0^T
\E\left[
|\tilde\xi_t^{1,N}|^k+|\tilde\zeta_t^{1,N}|^k+\|\tilde\chi_t^{1,N}\|^k
\right]\d t.
\end{align}
By the moment estimates for $U^1$ and $V^1$, the same argument used in the proof of Theorem~\ref{gf2}, together with \eqref{ap1}, gives
\begin{align}\label{wap1}
\lim_{N\to\infty}
\E\left[|\tilde\xi_t^{1,N}|^k+\|\tilde\chi_t^{1,N}\|^k\right]=0,\quad t\in[0,T].
\end{align}
If, in addition, ${\bf(H_4)}$ holds and $X_0,\eta\in L^q(\Omega\to\R^d,\F_0,\P)$, then the same argument, together with \eqref{ap2}, yields
\begin{align}\label{wap2}
\E\left[|\tilde\xi_t^{1,N}|^k+\|\tilde\chi_t^{1,N}\|^k\right]
\leq c_1\epsilon(N)^{\alpha\left( \ff{q-k}{\rho+\alpha q}\wedge 1\right)}
\end{align}
for some constant $c_1>0$.
On the other hand,  taking advantage of ${\bf(H_1)}$, \eqref{we3} and Jensen's inequality,
\begin{align}\label{wap3}
\E\left[|\tilde\zeta_t^{1,N}|^k\right]
&\leq \ff{c_2}{N}\sum_{j=1}^N\E|U_t^{j,N,1}|^k\leq c_3N^{-1}\|\eta\|_k^k,\quad t\in[0,T]
\end{align}
holds	for some constants $c_2,c_3>0$.

Therefore,	combining \eqref{xi-s}, \eqref{wap1} and \eqref{wap3}, and using dominated convergence, we obtain \eqref{ap3}. Moreover, \eqref{ap4} follows from \eqref{xi-s}, \eqref{wap2} and \eqref{wap3}, which completes the proof.
\end{proof}

\section{Application: Propagation of chaos for intrinsic derivative}\label{app}
Denote
$$
P_tf(\mu)=\E\left[f(X_t^{1})\right],\quad
P_t^{N}f(\mu)=\E\left[f(X_t^{1,N})\right],
\quad t\in[0,T],\,\mu\in\scr P_k,\,f\in\B_b(\R^d),
$$
where $X_t^{1}$ and $X_t^{1,N}$ are the unique solutions to \eqref{E1} and \eqref{E2} with
$\L_{X_0^1}=\L_{X_0^{1,N}}=\mu$, respectively.
For $\mu\in\scr P_k$ and $\phi\in T_{\mu,k}$, let
$$
\eta^i=\phi(X_0^i),\quad 1\leq i\leq N.
$$
Then $(\eta^i)_{1\leq i\leq N}$ are i.i.d. copies of $\eta=\phi(X_0)$. The intrinsic derivatives of $P_tf$ and $P_t^Nf$ at $\mu$ are given by
\begin{align*}
D_\phi^I(P_tf)(\mu)
=&\lim_{\vv\downarrow0}
\ff{P_tf\left(\mu\circ({\rm id}+\vv\phi)^{-1}\right)-P_tf(\mu)}{\vv}\\
=&\lim_{\vv\downarrow0}\ff1{\vv}
\E\left[f(X_t^{\phi,1,\vv})-f(X_t^1)\right],
\\
D_\phi^I(P_t^Nf)(\mu)
=&\lim_{\vv\downarrow0}
\ff{P_t^Nf\left(\mu\circ({\rm id}+\vv\phi)^{-1}\right)-P_t^Nf(\mu)}{\vv}\\
=&\lim_{\vv\downarrow0}\ff1{\vv}
\E\left[f(X_t^{\phi,1,N,\vv})-f(X_t^{1,N})\right],
\end{align*}
where $X_t^{\phi,1,\vv}$ and $X_t^{\phi,1,N,\vv}$ solve \eqref{E10} and \eqref{E20} with $\eta^i=\phi(X_0^i)$, $1\leq i\leq N$, respectively.

As an application of Theorem \ref{gf2} and \ref{gf3}, in this section, 
for any $f\in C_b^1(\R^d)$, by taking $N\to\8$, we prove that $D_\phi^I(P_t^Nf)(\mu)$ converges to  $D_\phi^I(P_tf)(\mu)$, with explicit convergence rates when $f\in C_b^2(\R^d)$. Since these intrinsic
derivatives with respect to initial distribution have Bismut-type representations, the result can also be
viewed as a finite-particle approximation of the Bismut formula.

%
We first provide a Bismut-type formula for \eqref{E2}. The proof is standard based on Lemma \ref{P3.1}, for details, see the proof of \cite[Theorem 2.1]{RW19}. We omit it here to avoid repetition.

\begin{lem}\label{bf}
Assume {$\bf(H_1)$} and {$\bf(H_2)$}. Then for any $f\in\B_b(\R^d)$ and $t\in(0,T]$, $P_t^Nf$ is intrinsically differentiable on $\scr P_k$. Moreover, for any $\mu\in\scr P_k$, $\phi\in T_{\mu,k}$, $t\in(0,T]$, and $g\in C^1([0,t])$ with $g_0=0$ and $g_t=1$,
\begin{equation}\label{BSMN}
\beg{split}
D_\phi^I(P_t^Nf)(\mu)
=&\sum_{i=1}^N
\E\bigg[
f(X_t^{1,N})
\int_0^t
\big\<(\si_s^*a_s^{-1})(X_s^{i,N})g_s'V_s^{\phi,i,N},\d W_s^i\big\>
\bigg],
\end{split}
\end{equation}
where $V_s^{\phi,i,N}$ is the unique solution to \eqref{v2} with $\eta^i=\phi(X_0^i)$.
\end{lem}
Actually,  \eqref{BSMN} provides a   particle-system approximation of the Bismut formula for the McKean--Vlasov SDE. This approximation will be justified in the following theorem.
\begin{thm}\label{bfa}
Assume {$\bf(H_1)$}, {$\bf(H_2)$} and {$\bf(H_3)$}.
Then for any   $f\in C_b^1(\R^d)$, $t\in(0,T]$, $\mu\in\scr P_k$,  $\phi\in T_{\mu,k}$, and
$g\in C^1([0,t])$ with $g_0=0$, $g_t=1$,
\begin{align*}
\lim_{N\to\8} D^I_{\phi} (P_t^N f)(\mu)=D^I_{\phi}(P_t f)(\mu)=\E\bigg[f(X_t^1)   \int_0^t \big\< \zeta_s^\phi,\ \d W_s^1\big\>\bigg],
\end{align*}
where   $X_t^{1}$ solves $\eqref{E1}$ for $i=1$, $\L_{X_0^{1}}=\mu$, and $V^{\phi,1}$ solves \eqref{v1} with $\eta^1=\phi(X_0^1)$,
\begin{align}\label{zp} 
\zeta_t^{\phi}:= (\si_t^*a_t^{-1})(X_t^{1})\left\{ g_t' V_t^{\phi,1}+ \big(\E \< D^L b_t(y,\cdot)(\L_{X_t^1})(X_t^1), g_t V_t^{\phi,1} \>\big)\big|_{y=X_t^1}	\right\}.
\end{align} 
If, in addition, {$\bf(H_4)$} holds, then for any $f\in C_b^2(\R^d)$, $t\in(0,T]$, $\mu\in\scr P_q$ and $\phi\in T_{\mu,q}$, there exists a constant $c>0$ independent of $N$, such that
\begin{align}\label{bma2}
\left|D_\phi^I(P_t^Nf)(\mu)-D_\phi^I(P_tf)(\mu)\right|^k
\leq
c\epsilon(N)^{\alpha \left(\ff{q-k}{\rho+\alpha q}\wedge1\right)}.
\end{align}
\end{thm}
\begin{proof}
By combining the Bismut formula in \cite{W23b} with the argument in   \cite{RW19}, under {$\bf(H_1)$}, {$\bf(H_2)$} and {$\bf(H_3)$}, we have
\begin{align*}
D_\phi^I(P_tf)(\mu)
=
\E\bigg[
f(X_t^1)\int_0^t\<\zeta_s^\phi,\d W_s^1\>
\bigg],
\quad f\in\B_b(\R^d),
\end{align*}
where $\zeta^\phi$ is defined by \eqref{zp}.

On the other hand,  the dominated convergence theorem implies that for any $\mu\in\scr P_k$, $f\in C^1_b(\R^d)$ and $t\in(0,T]$,
\begin{align*}
D_\phi^I(P_t^Nf)(\mu)
=
\E\left[\<\nn f(X_t^{1,N}),V_t^{\phi,1,N}\>\right],
\quad
D_\phi^I(P_tf)(\mu)
=
\E\left[\<\nn f(X_t^1),V_t^{\phi,1}\>\right].
\end{align*}
Hence, using H\"older's inequality,
\begin{equation}\label{bma}
\begin{aligned}
\left|
D_\phi^I(P_t^Nf)(\mu)-D_\phi^I(P_tf)(\mu)
\right|
\leq&
\|\nn f\|_\8
\E\left|V_t^{\phi,1,N}-V_t^{\phi,1}\right|\\
&+
\left(
\E|V_t^{\phi,1}|^2
\cdot
\E\left|\nn f(X_t^{1,N})-\nn f(X_t^1)\right|^2
\right)^{\ff12}.
\end{aligned}
\end{equation}
Since $\nn f$ is bounded and continuous, \eqref{ap01} yields
$$
\lim_{N\to\8}\E\left|\nn f(X_t^{1,N})-\nn f(X_t^1)\right|^2=0,\ \ t\in[0,T].
$$
Combining this with \eqref{ve1}, \eqref{ap1} and \eqref{bma}, we derive
$$
\lim_{N\to\8}D_\phi^I(P_t^Nf)(\mu)
=
D_\phi^I(P_tf)(\mu).
$$
Now assume further that {$\bf(H_4)$} holds. For any $\mu\in\scr P_q$, $\phi\in T_{\mu,q}$ and $f\in C_b^2(\R^d)$, from \eqref{ap02}, \eqref{ve1}, \eqref{ap2} and  \eqref{bma}, using Jensen's inequality, we obtain
\begin{align*}
\left|
D_\phi^I(P_t^Nf)(\mu)-D_\phi^I(P_tf)(\mu)
\right|
\leq&
\|\nn f\|_\8
\left(\E\left|V_t^{\phi,1,N}-V_t^{\phi,1}\right|^k\right)^{\ff1k}\\
&+
\|\nn^2 f\|_\8
\left(\E|V_t^{\phi,1}|^2
\cdot
\E|X_t^{1,N}-X_t^1|^2
\right)^{\ff12}\\
\leq&
c\epsilon(N)^{\ff\alpha k\left(\ff{q-k}{\rho+\alpha q}\wedge1\right)}
+
c\epsilon(N)^{\ff1k}.
\end{align*}
Since $\alpha\left(\ff{q-k}{\rho+\alpha q}\wedge1\right)\leq1$, the second term can be absorbed into the first one, then \eqref{bma2} follows. Therefore, we end the proof.
\end{proof}
\appendix
\section{Proof of Lemma \ref{poc} }\label{2.1}
\begin{enumerate}
\item [(1)] 
For any $t\in[0,T]$, $\boldsymbol{x}=(x^1,x^2,\cdots,x^N)
\in\R^{dN}$ for $x^i\in\R^d$, $1\leq i\leq N$ and $\mu^N_{\boldsymbol{x}}:=\ff 1 N\sum_{i=1}^N\delta_{x^i}$, denote
\begin{align*}
B^N_t(\boldsymbol{x})&:=\left(b_t(x^1,\mu^N_{\boldsymbol{x}}),b_t(x^2,\mu^N_{\boldsymbol{x}}),\cdots,b_t(x^N,\mu^N_{\boldsymbol{x}})\right),\\
\Sigma^N_t(\boldsymbol{x})&:=\operatorname{diag}\left(\si_t(x^1,\mu^N_{\boldsymbol{x}}),\si_t(x^2,\mu^N_{\boldsymbol{x}}),\cdots,\si_t(x^N,\mu^N_{\boldsymbol{x}})\right).
\end{align*}
Then under assumption {$\bf(H)$}, for any $N\geq1$, there exists a constant $C(N)>0$ such that for any $t\in[0,T]$, $\boldsymbol{x}=(x^1,x^2,\cdots,x^N)
,\boldsymbol{y}=(y^1,y^2,\cdots,y^N)
\in\R^{dN}$ for $x^i,y^i\in\R^d$, $1\leq i\leq N$, 	$$\sup_{t\in[0,T]}\left\{|B_t^N(\boldsymbol{0})|+\|\Sigma_t^N(\boldsymbol{0})\|\right\}<\8,$$ and
\begin{align*}
&2\big<\boldsymbol{x}-\boldsymbol{y},B^N_t(\boldsymbol{x})-B^N_t(\boldsymbol{y})\big\>^++\big\|\Sigma^N_t(\boldsymbol{x})-\Sigma^N_t(\boldsymbol{y})\big\|_{HS}^2\\&
\leq K\sum_{j=1}^N\Big(|x^j-y^j|^2+\W_k(\mu^N_{\boldsymbol{x}},\mu^N_{\boldsymbol{y}})^2\Big)\\&
\leq K\left(\|\boldsymbol{x}-\boldsymbol{y}\|_2^2+N^{1-\ff 2 k}\|\boldsymbol{x}-\boldsymbol{y}\|_k^2\right)
\leq C(N)\|\boldsymbol{x}-\boldsymbol{y}\|_{ 2}^2,
\end{align*}
where $\|\boldsymbol{x}-\boldsymbol{y}\|_p^p:=\sum_{j=1}^N|x^j-y^j|^{p}$, $p\geq1$, the second inequality follows from
\begin{align}\label{wk0}
\W_k(\mu^N_{\boldsymbol{x}},\mu^N_{\boldsymbol{y}})\leq \left(\ff 1 N\sum_{j=1}^N|x^j-y^j|^k\right)^{\ff 1 k}=N^{-\ff 1 k}\|\boldsymbol{x}-\boldsymbol{y}\|_k,
\end{align} and the final inequality is obtained from $$\|\boldsymbol{x}-\boldsymbol{y}\|_q\leq\|\boldsymbol{x}-\boldsymbol{y}\|_p\leq N^{ 1/p - 1/ q} \|\boldsymbol{x}-\boldsymbol{y}\|_q,\ \ 1\leq p\le q\le\8.$$ 
Noting that \eqref{E2} can be written as  a system in $(\R^d)^N$ with $\boldsymbol{X}_0=(X_0^1,X_0^2,\cdots,X_0^N)$:
\begin{align}
\d \boldsymbol{X}_t=B_t^N(\boldsymbol{X}_t)\d t+\Sigma_t^N(\boldsymbol{X}_t)\d \boldsymbol{W}_t,\ \ t\in[0,T],
\end{align}
where $\boldsymbol{X}=(X^1,X^2,\cdots,X^N)$ and $\boldsymbol{W}=(W^1,W^2,\cdots,W^N)$, \eqref{E2} is well-posed by the general SDE theory, see  \cite[Theorem 1.3.1]{RW24} for instance. 

Below, we prove the estimate \eqref{MM2}.  Using It\^o's formula and by {\bf (H)},  there exists a constant $c_1>0$ such that for any $1\le i\le N,\,t\in[0,T]$,\begin{align*}
\d |X_t^{i,N}|^2=&\Big\{2\<X_t^{i,N},b_t(X_t^{i,N},\hat\mu_t^N)\>+\|\si_t(X_t^{i,N},\hat\mu_t^N)\|_{HS}^2\Big\}\d t+2\<X_t^{i,N},\si_t(X_t^{i,N},\hat\mu_t^N)\d W_t^i\>\\
\le&c_1\left\{1+|X_t^{i,N}|^2+\left( \ff 1 N\sum_{j=1}^N|X_t^{j,N}|^k\right)^{\ff 2 k} \right\}\d t+2\<X_t^{i,N},\si_t(X_t^{i,N},\hat\mu_t^N)\d W_t^i\>.
\end{align*}
Then by  It\^o's formula, Jensen's and   Young's inequalities, for any $m\geq  k\geq 2$, it holds
\begin{align*}
\d |X_t^{i,N}|^m&\leq c_2(m)\left\{1+|X_t^{i,N}|^m+\left( \ff 1 N\sum_{j=1}^N|X_t^{j,N}|^k\right)^{\ff m k}\right\}\d t+\d N_t^i\\
&\leq c_2(m)\left\{1+|X_t^{i,N}|^m+\ff 1 N\sum_{j=1}^N|X_t^{j,N}|^m\right\}\d t+\d N_t^i,\ \ t\in[0,T], \,1\leq i\leq N
\end{align*}
for some constant $c_2(m)$, where $$\d \< N^i\>_t\leq c_2(m)|X_t^{i,N}|^{2m-2}\left(1+|X_t^{i,N}|^2+\Big(\ff 1 N\sum_{j=1}^N|X_t^{j,N}|^k\Big)^{\ff 2 k}\right)\d t.$$
Denote
$$G_t^i=\E\left[\sup_{s\in[0,t]}\left(1+|X_s^{i,N}|^m\right)\Big|\F_0\right],\ \ t\in[0,T],\,1\le i\le N.$$  
Invoking  the Burkholder--Davis--Gundy, the Jensen and  the Young inequalities, 
for any $1\leq i\leq N,\,t\in[0,T]$, we have
\begin{equation}\label{u1}
\begin{aligned}
G_t^i
\leq& \Big(1+|X_0^i|^{m}\Big)+c_3(m)\int_0^t\bigg\{G_s^i+\ff 1 N\sum_{j=1}^NG_s^j\bigg\}\,\d s\\&+\ff 1 2G_t^i+c_3(m)\int_0^t\E\bigg[|X_s^{i,N}|^{m-2}\bigg(\ff 1 N\sum_{j=1}^N|X_s^{j,N}|^k\bigg)^{\ff 2 k}\bigg|\F_0\bigg]\,\d s	\\
\leq& \Big(1+|X_0^i|^{m}\Big)+ c_4(m)\int_0^t \Big(G_s^i+\ff 1 N\sum_{j=1}^N G_s^j\Big)\,\d s+\ff 1 2G_t^i\end{aligned}
\end{equation}
for some constants $c_3(m),c_4(m)>0$.
Then it follows
\begin{align*}
\ff 1 N\sum_{j=1}^NG_t^j\leq&\ff 2 N\sum_{j=1}^N(1+|X_0^j|^m)+4c_4(m)\int_0^t\ff { 1}{N}\sum_{j=1}^NG_s^j\,\d s, \ \ t\in[0,T],
\end{align*}
which together with Gr\"onwall's inequality yields
\begin{align*}
\ff 1 N\sum_{j=1}^NG_t^j\leq 2\e^{4c_4(m)T}\left(1+\ff 1 N\sum_{j=1}^N|X_0^j|^m\right),\ \ t\in[0,T].
\end{align*}
Substituting this into \eqref{u1} and using Gr\"onwall's inequality again, we can find a constant $c_5(m)>0$ such that
\begin{align}\label{g}
G_t^i\leq c_5(m)\left(1+|X_0^i|^m+\ff 1 N\sum_{j=1}^N|X_0^j|^m\right),\ \ t\in[0,T],\,1\le i\le N,
\end{align} which implies \eqref{MM2}.
\item [(2)] 
Since $(X_t^i)_{1\leq i\leq N}$ are i.i.d., recalling that $\tt\mu_t^N:=\ff 1 N\sum_{j=1}^N\dd_{X_t^j}$, we have
\begin{align}\label{ap'1}
\lim_{N\to\8}\sup_{t\in[0,T]}\E\left[\W_k(\tt\mu_t^N,\mu_t)^k\right]=0.
\end{align}
Indeed,  \eqref{ap'1} follows from \eqref{MM1}, the law of large numbers for empirical measures on Polish spaces and \cite[Theorem 5.5]{CD}.

On the other hand, by {$\bf(H)$} and It\^o's formula, 
\begin{align*}
\d |X_t^{i,N}-X_t^i|^2=&\Big\{2\<X_t^{i,N}-X_t^i,b_t(X_t^{i,N},\hat\mu_t^N)-b_t(X_t^i,\mu_t)\>\\&+\|\si_t(X_t^{i,N},\hat\mu_t^N)-\si_t(X_t^i,\mu_t)\|_{HS}^2\Big\}\d t+\d M_t^i\\
\le&K\left(|X_t^{i,N}-X_t^i|^2+\W_k(\hat\mu_t^N,\mu_t)^2\right)\d t+\d M_t^i,\ \ 1\le i\le N,\,t\in[0,T],
\end{align*}
where $\d M_t^i:=2\<X_t^{i,N}-X_t^i,(\si_t(X_t^{i,N},\hat\mu_t^N)-\si_t(X_t^i,\mu_t))\d W_t^i\>$. Then using  It\^o's formula once more, and applying Young's inequality, there exists a constant $c_6(k)>0$ such that for any $t\in[0,T],\,1\le i\le N$,
\begin{align}\label{it1}
\d |X_t^{i,N}-X_t^i|^k\leq c_6(k)\left(|X_t^{i,N}-X_t^i|^k+\W_k(\hat\mu_t^N,\mu_t)^k\right)\d t+ \d \tt M_t^i,
\end{align}
where $$\d \<\tt M^i\>_t\leq c_6(k)|X_t^{i,N}-X_t^i|^{2k-2}\left(|X_t^{i,N}-X_t^i|^2+\W_k(\hat\mu_t^N,\mu_t)^2\right)\d t.$$

By the triangle inequality and \eqref{wk0}, for some constant $c_0(k)>0$,
\begin{equation}\label{wkt}
\begin{aligned}
\W_k(\hat\mu_t^N,\mu_t)^k
&\le c_0(k)\W_k(\hat\mu_t^N,\tt\mu_t^N)^k+c_0(k)\W_k(\tt\mu_t^N,\mu_t)^k\\
&\le c_0(k)\ff1N\sum_{j=1}^N|X_t^{j,N}-X_t^j|^k+c_0(k)\W_k(\tt\mu_t^N,\mu_t)^k.
\end{aligned}
\end{equation}
Hence, taking advantage of   the Burkholder--Davis--Gundy and the Young inequalities, and by the identical distribution of the coupled pairs $(X_t^{i,N},X_t^i)_{1\leq i\leq N}$, 
it follows for any $t\in[0,T]$ and some constants $c_7(k), c_8(k)>0$ that
\begin{equation}\label{it2}
\begin{aligned}
&\E\left[\sup_{s\in[0,t]}|X_s^{i,N}-X_s^i|^k\right]\leq  c_6(k)\int_0^t\left(\E|X_s^{i,N}-X_s^i|^k+\E\W_k(\hat\mu_s^N,\mu_s)^k\right)\,\d s\\&\quad\quad\quad+\E\left(c_6(k) \int_0^t|X_s^{i,N}-X_s^i|^{2k-2}\left(|X_s^{i,N}-X_s^i|^2+\W_k(\hat\mu_s^N,\mu_s)^2\right)\,\d s\right)^{\ff 1 2} \\
&\leq  c_7(k)\int_0^t\left(\E|X_s^{i,N}-X_s^i|^k+\E\W_k(\hat\mu_s^N,\mu_s)^k\right)\,\d s\\&\quad\quad\quad+\ff 1 2\E\left[\sup_{s\in[0,t]}|X_s^{i,N}-X_s^i|^k\right]+c_7(k)\int_0^t\E\left[|X_s^{i,N}-X_s^{i}|^{k-2}\W_k(\hat\mu_s^N,\mu_s)^2\right]\,\d s\\
&\leq c_8(k)\int_0^t\E|X_s^{i,N}-X_s^i|^k\,\d s+
\ff 1 2\E\left[\sup_{s\in[0,t]}|X_s^{i,N}-X_s^i|^k\right]\\&\quad\quad\quad+c_8(k)\int_0^t\E\W_k(\tt\mu_s^N,\mu_s)^k\,\d s,\ \ 1\leq i\leq N.	\end{aligned}
\end{equation}
Combining this with Gr\"onwall's inequality implies that
\begin{align}\label{ak1}
\E\left[\sup_{s\in[0,t]}|X_s^{i,N}-X_s^i|^k\right]\leq c_9(k)\sup_{s\in[0,t]}\E\W_k(\tt\mu_s^N,\mu_s)^k,\ \ 1\le i\le N,\,t\in[0,T]
\end{align}
for some constant $c_9(k)>0$, which tends to $0$ by \eqref{ap'1}. In addition, together  with \eqref{wkt}, we can also obtain that
\begin{align*}
\lim_{N\to\8}\sup_{t\in[0,T]}\E\left[\W_k(\hat\mu_t^N,\mu_t)^k\right]=0.
\end{align*} Thus, we derive the desired estimate \eqref{ap01}.
\item [(3)] 	If further assume that $\mu=\L_{X_0}\in\scr P_{q}$ for some $q>k$, according to  \cite[Theorem 1]{FG},
\begin{align}\label{wk2}
\sup_{t\in[0,T]}\E\W_k(\tt\mu_t^N,\mu_t)^k\leq c(q,k)\epsilon(N)
\end{align} for some constant $c(q,k)>0$. Then \eqref{ap02} follows from \eqref{wkt}, \eqref{ak1}, \eqref{wk2}. Therefore, we complete the proof.
\end{enumerate}
\section{Extensions of Lemmas \ref{gf1} and \ref{P3.1}}\label{df}
The conclusions of  Lemmas \ref{gf1} and \ref{P3.1} continue to hold when $\bf(H_1)$ is replaced by the following weaker condition.
\begin{enumerate}\item[$\bf (H_B)$] 
There exist constants $K>0$ and $q\geq k\geq2$ such that \eqref{SPPn} holds, and
\begin{enumerate}
\item [$(B_1)$]  For any $(t,x,\mu)\in [0,T]\times \R^d\times \scr P_q$,
\begin{align}\label{l1}
\|\nn b_t(x,\mu)\|\le K\left\{1+|x|^{\ff {q-k}k}+\|\mu\|_q^{\ff {q-k} k}\right\},
\end{align}
where $\|\mu\|_q^q:=\mu(|\cdot|^q)$.
\item [$(B_2)$] For any $t\in[0,T]$, $b_t,\si_t\in\D_q$ such that 
\begin{align}\label{l2}
\sup_{(t,x,\mu)\in[0,T]\times\R^d\times\scr P_q}\left\{\|D^Lb_t(x,\cdot)(\mu)\|_{L^{k^*}(\mu)}+\|D^L\si_t(x,\cdot)(\mu)\|_{L^{k^*}(\mu)}\right\}<\8,
\end{align}
and for any $(t,x,v,\mu)\in[0,T]\times \R^d\times\R^d\times\scr P_q$,
\begin{align*}
2\< v,\nn_v b_t(\cdot,\mu)(x)\>^++\|\nn_v\si_t(\cdot,\mu)(x)\|_{HS}^2\le K|v|^2.	\end{align*} 
\end{enumerate} 
\end{enumerate}
\begin{lem}\label{b}
Replace assumption 
$\bf(H_1)$ in Lemmas \ref{gf1} and \ref{P3.1} by 
$\bf(H_B)$, while keeping all the other assumptions unchanged. Then for any $X_0,\eta\in L^q(\Omega\to\R^d,\F_0,\P)$, the conclusions of Lemmas \ref{gf1} and \ref{P3.1} remain valid.
\end{lem}
{\begin{rem}
Assumption ${\bf(H_B)}$ does not require $b$ to be globally Lipschitz in $x$. When $q>k$, it allows dissipative drifts with unbounded spatial derivatives, for example, for any $\beta\le {\frac{q-k}{2k}}$,
$$
b_t(x,\mu)=-(1+|x|^2)^\beta x,\ \ (t,x,\mu)\in [0,T]\times\R^d\times\scr P_q.
$$
Such drifts satisfy \eqref{l1} and the one-sided condition in $(B_2)$, but
are not globally Lipschitz in $x$. Hence ${\bf(H_B)}$ covers a class of dissipative
drifts with polynomially growing derivatives.
\end{rem}}
\begin{proof}[Proof of Lemma \ref{b}] Since the remaining arguments are similar, we only provide the proof of \eqref{vi2} to avoid repetition.
Note that \eqref{MM2} yields that $\hat\mu_t^N\in\scr P_q$,  almost surely, then it follows from \eqref{l2}  that 
\begin{equation}\label{d1}
\begin{aligned}
&\ff 1 N\sum_{j=1}^N\Big|D^Lb_t(X_t^{i,N},\cdot)(\hat\mu_t^N)(X_t^{j,N})\Big|^{k^*}
=\|D^Lb_t(X_t^{i,N},\cdot)(\hat\mu_t^N)\|_{L^{k^*}(\hat\mu_t^N)}^{k^*}\le M,\ \ \text{a.s.}
\end{aligned}
\end{equation} for some constant $M>0$. Thus H\"older's inequality implies that
\begin{equation}\label{d2}
\begin{aligned}
\left|\ff 1 N\sum_{j=1}^N\Big\<D^Lb_t(X_t^{i,N},\cdot)(\hat\mu_t^N)(X_t^{j,N}),V_t^{j,N}\Big\>\right|\le M^{\ff 1 {k^*}}\left(\ff 1 N\sum_{j=1}^N|V_t^{j,N}|^k\right)^{\ff 1 k}.
\end{aligned}
\end{equation}
Similarly, \eqref{d1} and \eqref{d2}  hold for $\si$ replacing $b$.
Along with \eqref{d1}, \eqref{d2} and the arguments used in the original proof of \eqref{ve2}, we obtain that for any $m\geq k$, there exists a constant $c_1>0$ such that
\begin{align}\label{ve2'}
\E\left[\sup_{t\in[0,T]}|V_t^{i,N}|^m\bigg|\F_0\right]\leq c_1\left( |\eta^i|^m+\ff 1 N\sum_{j=1}^N|\eta^j|^m\right),\ \ 1\le i\le N.
\end{align} 
It remains to prove that
\begin{align*}
\lim_{\varepsilon\downarrow0}
\max_{1\leq i\leq N}
\E\left[
\sup_{t\in[0,T]}
\left|
\ff{X_t^{i,N,\varepsilon}-X_t^{i,N}}{\varepsilon}
-
V_t^{i,N}
\right|^k
\right]=0.
\end{align*}
Set
$$
\widetilde V_t^{i,N,\varepsilon}
:=
\ff{X_t^{i,N,\varepsilon}-X_t^{i,N}}{\varepsilon},
\quad 1\leq i\leq N,\ \varepsilon\in(0,1].
$$
Using \eqref{wk0}, assumption {$(B_2)$}, It\^o's formula, Jensen's and Young's inequalities, we find a constant $c_2>0$ such that
\begin{align*}
\d |\widetilde V_t^{i,N,\varepsilon}|^q
&\leq
c_2\left\{
|\widetilde V_t^{i,N,\varepsilon}|^q
+\varepsilon^{-q}\W_k(\hat\mu_t^{N,\varepsilon},\hat\mu_t^N)^q
\right\}\d t+\d M_t^i\\
&\leq
c_2\left\{
|\widetilde V_t^{i,N,\varepsilon}|^q
+\ff1N\sum_{j=1}^N|\widetilde V_t^{j,N,\varepsilon}|^q
\right\}\d t+\d M_t^i, \ \ 1\le i\le N,\,t\in[0,T],
\end{align*}
where $\widetilde V_0^{i,N,\varepsilon}=\eta^i$ and
\begin{align*}
\d\<M^i\>_t
\leq c_2|\widetilde V_t^{i,N,\varepsilon}|^{2q-2}
\left\{
|\widetilde V_t^{i,N,\varepsilon}|^2
+
\left(\ff1N\sum_{j=1}^N|\widetilde V_t^{j,N,\varepsilon}|^k\right)^{\ff2k}
\right\}\d t.
\end{align*}
By the Burkholder--Davis--Gundy, Jensen's and Young's inequalities, there exists a constant $c_3>0$ such that
\begin{align*}
\E\left[\sup_{s\in[0,t]}|\widetilde  V_s^{i,N,\vv}|^q\right]\le& \|\eta\|_q^q+c_3\int_0^t\E|\widetilde  V_s^{i,N,\vv}|^q\,\d s+\ff 1 2	\E\left[\sup_{s\in[0,t]}|\widetilde  V_s^{i,N,\vv}|^q\right]\\+&\ff 1 N\sum_{j=1}^N
\int_0^t\E|\widetilde  V^{j,N,\vv}_s|^q\,\d s,\ \ 1\le i\le N,\,t\in[0,T],
\end{align*}where $\|\eta\|_q^q:=\E|\eta|^q$. Repeating the argument leading to \eqref{ak1}, 
\begin{align}\label{tv1}
\max_{1\leq i\leq N}
\E\left[\sup_{t\in[0,T]}|\widetilde V_t^{i,N,\varepsilon}|^q\right]
\leq c_4\|\eta\|_q^q<\infty
\end{align}
holds for some constant $c_4>0$.

Subsequently, for any $t\in[0,T],\,\t\in[0,1]$, denote $$\gamma_t^{i,N,\t}=X_t^{i,N}+\t(X_t^{i,N,\vv}-X_t^{i,N}), \ 1\le i\le N, \ \ \hat\mu_t^{N,\t,\gamma}=\ff 1 N\sum_{j=1}^N\delta_{\gamma_t^{j,N,\t}},$$ and for any $t\in[0,T],\,1\le i\le N$, we write
\begin{equation}\label{tv2}
\begin{aligned}
\d \widetilde V_t^{i,N,\varepsilon}
=&
\nn_{\widetilde V_t^{i,N,\varepsilon}}b_t(X_t^{i,N},\hat\mu_t^N)\d t
+\ff1N\sum_{j=1}^N
\Big\<D^Lb_t(X_t^{i,N},\cdot)(\hat\mu_t^N)(X_t^{j,N}),
\widetilde V_t^{j,N,\varepsilon}\Big\>\d t\\
&+\alpha_t^{i,\varepsilon}\d t+\beta_t^{i,\varepsilon}\d W_t^i+
\bigg\{
\nn_{\widetilde V_t^{i,N,\varepsilon}}\si_t(X_t^{i,N},\hat\mu_t^N)\\
&
+\ff1N\sum_{j=1}^N
\Big\<D^L\si_t(X_t^{i,N},\cdot)(\hat\mu_t^N)(X_t^{j,N}),
\widetilde V_t^{j,N,\varepsilon}\Big\>
\bigg\}\d W_t^i,
\end{aligned}
\end{equation}
where through Lemma \ref{de}, \begin{align*}
&\alpha^{i,\vv}_t:=\ff {b_t(X_t^{i,N,\vv},\hat\mu_t^{N})-b_t(X_t^{i,N},\hat\mu_t^N)}{\vv}-\nn_{\widetilde  V_t^{i,N,\vv}}b_t(X_t^{i,N},\hat\mu_t^N)\\
&\quad+\ff {b_t(X_t^{i,N,\vv},\hat\mu_t^{N,\vv})-b_t(X_t^{i,N,\vv},\hat\mu_t^N)}{\vv}-\ff 1 N\sum_{j=1}^N\Big\<D^Lb_t(X_t^{i,N},\cdot)(\hat\mu_t^N)(X_t^{j,N}),\widetilde  V_t^{j,N,\vv}\Big\>\\
&=\int_0^1 \left\{\nn_{\widetilde  V_t^{i,N,\vv}}b_t(\gamma_t^{i,N,\t},\hat\mu_t^N)-\nn_{\widetilde  V_t^{i,N,\vv}}b_t(X_t^{i,N},\hat\mu_t^N)\right\}\,\d\t\\
&\quad+\ff 1 N\sum_{j=1}^N\int_0^1\Big\<\widetilde  V_t^{j,N,\vv},D^Lb_t(X_t^{i,N,\vv},\cdot)(\hat\mu_t^{N,\t,\gamma})(\gamma_t^{j,N,\t})-D^Lb_t(X_t^{i,N},\cdot)(\hat\mu_t^N)(X_t^{j,N})\Big\>\d \t,	\end{align*}and
\begin{align*}
&\beta^{i,\vv}_t:=\ff {\si_t(X_t^{i,N,\vv},\hat\mu_t^{N})-\si_t(X_t^{i,N},\hat\mu_t^N)}{\vv}-\nn_{\widetilde  V_t^{i,N,\vv}}\si_t(X_t^{i,N},\hat\mu_t^N)\\
&\quad+\ff {\si_t(X_t^{i,N,\vv},\hat\mu_t^{N,\vv})-\si_t(X_t^{i,N,\vv},\hat\mu_t^N)}{\vv}-\ff 1 N\sum_{j=1}^N\Big\<D^L\si_t(X_t^{i,N},\cdot)(\hat\mu_t^N)(X_t^{j,N}),\widetilde  V_t^{j,N,\vv}\Big\>\\
&=\int_0^1 \left\{\nn_{\widetilde  V_t^{i,N,\vv}}\si_t(\gamma_t^{i,N,\t},\hat\mu_t^N)-\nn_{\widetilde  V_t^{i,N,\vv}}\si_t(X_t^{i,N},\hat\mu_t^N)\right\}\,\d\t\\
&\quad+\ff 1 N\sum_{j=1}^N\int_0^1\Big\<\widetilde  V_t^{j,N,\vv},D^L\si_t(X_t^{i,N,\vv},\cdot)(\hat\mu_t^{N,\t,\gamma})(\gamma_t^{j,N,\t})-D^L\si_t(X_t^{i,N},\cdot)(\hat\mu_t^N)(X_t^{j,N})\Big\>\d \t.
\end{align*} By the continuity contained in the definition  of $\D_q$, we can repeat   the proof of \cite[(2.18), (4.9)]{W23b}  to derive that for every fixed $t\in(0,T]$,
\begin{align}\label{ab0}
\lim_{\varepsilon\downarrow0}
\max_{1\leq i\leq N}
\left\{
|\alpha_t^{i,\varepsilon}|+\|\beta_t^{i,\varepsilon}\|
\right\}
=0,\quad \P\text{-a.s.}
\end{align} 

Additionally,  {$(B_1)$}, H\"older's inequality, \eqref{wk0} and \eqref{d1} imply that 
\begin{align*}
|\alpha_t^{i,\varepsilon}|^k+\|\beta_t^{i,\varepsilon}\|^k
\leq&
\ff{c_5}{N}\sum_{j=1}^N
\left(|\widetilde V_t^{j,N,\varepsilon}|^k
+|\widetilde V_t^{i,N,\varepsilon}|^k\right)\times
\bigg\{
1+|X_t^{i,N}|^{q-k}+|X_t^{i,N,\varepsilon}|^{q-k}
\\
&+|X_t^{j,N}|^{q-k}+|X_t^{j,N,\varepsilon}|^{q-k}
+\bigg(\ff1N\sum_{j=1}^N\left(|X_t^{j,N}|+|X_t^{j,N,\vv}-X_t^{j,N}|\right)^q\bigg)^{\ff{q-k}{q}}
\bigg\}
\end{align*}
for some constant $c_5>0$. Since $X_0,\eta\in L^q(\Omega\to\R^d,\F_0,\P)$, then by \eqref{MM2}, \eqref{tv1} and H\"older's inequality, we can obtain
\begin{align*}
\sup_{\varepsilon\in(0,1]}
\max_{1\leq i\leq N}
\E\left[\sup_{t\in[0,T]}
\left\{
|\alpha_t^{i,\varepsilon}|^k+\|\beta_t^{i,\varepsilon}\|^k
\right\}\right]
<\infty.
\end{align*}
Combining \eqref{ab0} with the above uniform integrability and Vitali's convergence theorem, it follows
\begin{align}\label{ab1}
\lim_{\vv\downarrow0}\max_{1\le i\le N}\E\left[\int_0^T\left\{|\alpha^{i,\vv}_t|^k+\|\beta_t^{i,\vv}\|^k\right\}\,\d t\right]=0.
\end{align}
Now for any $t\in[0,T]$, $1\le i\le N$ and $\vv>0$, set
$$
\Lambda_t^{i,N,\varepsilon}
=
\widetilde V_t^{i,N,\varepsilon}-V_t^{i,N},
$$
then subtracting the equation for $V_t^{i,N}$ from \eqref{tv2},  we get
\begin{equation*}
\begin{aligned}
\d \Lambda_t^{i,N,\varepsilon}
=&
\left\{
\nn_{\Lambda_t^{i,N,\varepsilon}}b_t(X_t^{i,N},\hat\mu_t^N)
+\ff1N\sum_{j=1}^N
\Big\<D^Lb_t(X_t^{i,N},\cdot)(\hat\mu_t^N)(X_t^{j,N}),
\Lambda_t^{j,N,\varepsilon}\Big\>
\right\}\d t\\
&+\alpha_t^{i,\varepsilon}\d t+\beta_t^{i,\varepsilon}\d W_t^i+
\bigg\{
\nn_{\Lambda_t^{i,N,\varepsilon}}\si_t(X_t^{i,N},\hat\mu_t^N)\\
&
+\ff1N\sum_{j=1}^N
\Big\<D^L\si_t(X_t^{i,N},\cdot)(\hat\mu_t^N)(X_t^{j,N}),
\Lambda_t^{j,N,\varepsilon}\Big\>
\bigg\}\d W_t^i
\end{aligned}
\end{equation*}with $\Lambda_0^{i,N,\varepsilon}=0$.
By {$(B_2)$},  It\^o's formula, \eqref{d2}, H\"older's and Young's inequalities, there exists a constant $c_6>0$ such that
\begin{align*}
\d|\Lambda_t^{i,N,\varepsilon}|^k
\leq&
c_6\left\{
|\Lambda_t^{i,N,\varepsilon}|^k
+\ff1N\sum_{j=1}^N|\Lambda_t^{j,N,\varepsilon}|^k
+|\alpha_t^{i,\varepsilon}|^k+\|\beta_t^{i,\varepsilon}\|^k
\right\}\d t+\d M_t^i,
\end{align*}
where
\begin{align*}
\d\<M^i\>_t
\leq&
c_6|\Lambda_t^{i,N,\varepsilon}|^{2k-2}
\bigg(
|\Lambda_t^{i,N,\varepsilon}|^2
+\|\beta_t^{i,\varepsilon}\|^2\\
&\quad+
\bigg\|
\ff1N\sum_{j=1}^N
\Big\<D^L\si_t(X_t^{i,N},\cdot)(\hat\mu_t^N)(X_t^{j,N}),
\Lambda_t^{j,N,\varepsilon}\Big\>
\bigg\|^2
\bigg)\d t.
\end{align*}
Due to \eqref{d2} for $\si$ replacing $b$, the H\"older and the Burkholder--Davis--Gundy inequalities, we obtain
\begin{equation*}
\begin{aligned}\label{lm}
\E\left[\sup_{s\in[0,t]}|\Lambda^{i,N,\vv}_s|^k\right]\leq& c_7\int_0^t\E\left\{|\Lambda^{i,N,\vv}_s|^k+\ff 1 N\sum_{j=1}^N|\Lambda^{j,N,\vv}_s|^k\right\}\,\d s\\
&+\ff 1 2	\E\left[\sup_{s\in[0,t]}|\Lambda^{i,N,\vv}_s|^k\right]+c_7\int_0^t\E\left\{|\alpha^{i,\vv}_s|^k+\|\beta_s^{i,\vv}\|^k\right\}\,\d s
\end{aligned}
\end{equation*}
holds  for some $c_7>0$ and for any $t\in[0,T]$, $1\le i\le N$. 
Then repeating the argument leading to \eqref{g} once more, it follows for some constant $c_8>0$,
\begin{align*}
\max_{1\le i\le N}\E\left[\sup_{t\in[0,T]}|\Lambda^{i,N,\vv}_t|^k\right]\leq c_8\max_{1\le i\le N}\E\left[\int_0^T\left\{|\alpha^{i,\vv}_t|^k+\|\beta_t^{i,\vv}\|^k\right\}\,\d t\right], 
\end{align*}
which tends to $0$ as $\vv\downarrow0$ by \eqref{ab1}.  This completes the proof.
\end{proof}
\section*{Acknowledgments}
The author would like to thank Professor Feng-Yu Wang for useful conversations and corrections.

\end{document}